\documentclass[11pt]{amsart}
\usepackage{amsmath,amssymb,amsthm,color,enumerate,enumitem}
\usepackage{hyperref} 
\usepackage[normalem]{ulem} 
\usepackage{xcolor}
\usepackage{mathtools}

\newtheorem{thm}{Theorem}
\newtheorem{thmx}{Theorem}

\newtheorem{prop}[thm]{Proposition}
\newtheorem{cor}[thm]{Corollary}
\newtheorem{lemma}[thm]{Lemma}

\newtheorem{question}[thm]{Question}
\newtheorem{problem}[thm]{Problem}
\theoremstyle{definition}
\newtheorem{defin}[thm]{Definition}
\newtheorem{convention}[thm]{Convention}

\theoremstyle{remark}
\newtheorem{remark}[thm]{Remark}

\newcommand{\en}{\mathbb N}
\newcommand{\qe}{\mathbb Q}
\newcommand{\setsep}{:\;}

\newcommand{\A}{\mathcal A}
\newcommand{\F}{\mathcal F}

\newcommand{\B}{\mathcal B}

\newcommand{\K}{\mathcal K}
\newcommand{\C}{\mathcal C}

\newcommand{\M}{\mathcal M}

\newcommand{\T}{\mathcal T}
\newcommand{\PP}{\mathcal P}

\newcommand{\Id}{\operatorname{Id}}

\newcommand{\Span}{\operatorname{span}}
\newcommand{\closedSpan}{\overline{\operatorname{span}}}

\def \suppt {\operatorname{suppt}}

\def \dens {\operatorname{dens}}

\begin{document}

\title[Characterizations using skeletons and SPRI]{Characterizations of weakly $\mathcal{K}$-analytic and Va\v{s}\'ak spaces using projectional skeletons and separable PRI}

\author[C. Correa]{Claudia Correa}
\address[C. Correa]{Centro de Matem\'atica, Computa\c c\~ ao e Cogni\c c\~ ao, Universidade Federal do ABC, Avenida dos Estados, 5001, Santo Andr\' e, Brasil}
\email{claudiac.mat@gmail.com, claudia.correa@ufabc.edu.br}

\author[M. C\'uth]{Marek C\'uth}
\address[M. C\'uth]{Faculty of Mathematics and Physics, Department of Mathematical Analysis\\
Charles University\\
186 75 Praha 8\\
Czech Republic}
\email{cuth@karlin.mff.cuni.cz}
\address[M. C\'uth]{Institute of Mathematics of the Czech Academy of Sciences, \v{Z}itn\'a 25, 115 67 Prague 1, Czech Republic}

\author[J. Somaglia]{Jacopo Somaglia}
\address[J. Somaglia]{Dipartimento di Matematica ``F. Enriques''\\ Universit\`a degli Studi di Milano\\ Via Cesare Saldini 50\\ 20133 Milano\\ Italy}
\email{jacopo.somaglia@unimi.it}


\begin{abstract}We find characterizations of Va\v{s}\'ak spaces and weakly $\mathcal{K}$-analytic spaces using the notions of separable projectional resolution of the identity (SPRI) and of projectional skeleton. This in particular addresses a recent challenge suggested by M. Fabian and V. Montesinos in \cite{FM18}. Our method of proof also gives similar characterizations of WCG spaces and their subspaces (some aspects of which were known, some are new).
Moreover we show that for countably many projectional skeletons $\{\mathfrak{s}_n: n \in \omega\}$ on a Banach space inducing the same set, there exists a projectional skeleton on the space (indexed by ranges of the corresponding projections) which is isomorphic to a subskeleton of each $\mathfrak{s}_n$, $n \in \omega$.

\end{abstract}

\thanks{C. Correa has been partially supported by Funda\c c\~ao de Amparo \`a Pesquisa do Estado de S\~ao Paulo (FAPESP) grants 2018/09797-2 and 2019/08515-6. M.~C\'uth has been supported by Charles University Research program No. UNCE/SCI/023, GA\v{C}R project 19-05271Y and RVO: 67985840.  J.~Somaglia has been supported  by Universit\`a degli Studi di Milano and by Gruppo Nazionale per l'Analisi Matematica, la Probabilit\`a e le loro Applicazioni (GNAMPA) of Istituto Nazionale di Alta Matematica (INdAM), Italy.}

\maketitle

\section{Introduction}

Banach spaces with a projectional skeleton form quite a rich class of Banach spaces (most importantly nonseparable ones), which share some structural properties with separable spaces. Examples of Banach spaces admitting a projectional skeleton come from several areas of functional analysis such as reflexive spaces, WCG spaces, duals of Asplund spaces \cite{kubis09}, preduals of Von Neumann algebras \cite{BHK}, preduals of JBW$^*$ -triples \cite{BHKPP} and several examples of $\C(K)$ spaces \cite{So20}.

One of the key aspects behind the definition of a projectional skeleton is the possibility to build a \emph{projectional resolution of the identity} (PRI) on a Banach space, which is a long sequence of bounded linear projections onto subspaces of smaller densities that enables us to use effectively transfinite induction, see e.g. \cite[Sections 6-7]{FabRedBook} for the definition of PRI and applications. In 1968, D. Amir and J. Lindenstrauss \cite{AL68} proved that a PRI exists in every \emph{weakly compactly generated space} (WCG).
Later there were found several superclasses of WCG spaces, for which a PRI exists as well. Those include the following classes of Banach spaces (where the inclusions are always strict, SWCG denotes the class of subspaces of WCG spaces and WLD the class of weakly Lindel\"of determined spaces)
\[WCG\subset SWCG \subset \text{weakly $\mathcal{K}$-analytic} \subset \text{Va\v{s}\'ak}\subset WLD\subset \text{Plichko}.\]

A space for which the existence of a PRI is known, but is not Plichko is $\C([0,\kappa])$ for regular cardinals $\kappa\geq \omega_2$, see \cite{kal02}. We refer the reader e.g. to monographs \cite{DGZ, FabRedBook, HMVZ}, where more information may be found about the classes of Banach spaces mentioned above.\\ In 2009, W. Kubi\'{s} \cite{kubis09} introduced the notion of projectional skeletons. The class of Banach spaces that admit a projectional skeleton is contained in the class of Banach spaces admitting a PRI and it  contains all Plichko spaces as well as the spaces $\C([0,\kappa])$, for any ordinal $\kappa$. Nowadays, it seems that the class of Banach spaces with a projectional skeleton contains all the important classes of Banach spaces admitting PRI, which enables us to provide a uniform treatment for several known results proved previously for each subclass separately. Quite recently, this feature was somehow precised in \cite[Theorem 1.1]{kal20}, where it was proved that spaces admitting a 1-projectional skeleton form a ``$\PP$-class'' (that is, they admit PRI with a certain property). Consequently, they admit a LUR renorming and a strong Markushevich basis, see  \cite{Ziz84} and \cite[Theorem 5.1]{HMVZ}.

Recently, several authors found characterizations of some classes of Banach spaces in terms of projectional skeletons. Those include characterizations of Plichko spaces, WLD spaces, Asplund spaces, WLD Asplund spaces, WCG spaces and SWCG spaces, see \cite{CF16, CF17, FM18, kubis09}. From the classes mentioned previously, characterizations of weakly $\mathcal{K}$-analytic spaces and of Va\v{s}\'ak spaces in terms of projectional skeleton were not known and actually those two were proposed to tackle as a ``challenge" in \cite{FM18}. We accepted this challenge and solved the problem, which might be considered as the main outcome of our paper.\\
Our characterization of weakly $\mathcal{K}$-analytic Banach spaces is the following (see Sections~\ref{sec:prelim} and \ref{sec:mainResults} for the relevant definitions).

\begin{thmx}\label{thm:main}
Let $X$ be a Banach space and put $\kappa:=\dens(X)$. Then the following conditions are equivalent.
\begin{enumerate}[label = (\roman*)]
    \item\label{it:Xwanalytic} $X$ is a weakly $\mathcal{K}$-analytic space.
    \item\label{it:PSshrinkwanalytic} There exist a projectional skeleton $\mathfrak{s}=(P_s)_{s\in\Gamma}$ on $X$ and a family of non-empty sets $\{A_t\subset B_X\colon t\in\omega^{<\omega}\}$ satisfying the following conditions:
\begin{enumerate}
    \item\label{it:lindenscountsupportwanalytic} the set $\A := \bigcup_{\sigma\in \omega^{\omega}}\bigcap_{i\ge1} A_{\sigma|i}$ is linearly dense in $X$ and countably supports $X^*$; moreover $A_\emptyset=\A$ and $A_t \subset \A$, for every $t \in \omega^{<\omega}$;
    \item for every $\varepsilon >0$, $x^*\in X^*$, and $\sigma\in \omega^{\omega}$, there exists $i\in \omega$ such that  $\mathfrak{s}$ is $(A_{\sigma|i},\varepsilon)$-shrinking in $x^*$.
\end{enumerate}
\item\label{it:SPRIwanalytic} There exist a SPRI $(Q_{\alpha})_{\alpha\leq \kappa}$ in $X$ and a family of non-empty sets $\{A_t \subset B_X\colon t \in\omega^{<\omega}\}$ satisfying the following conditions:
\begin{enumerate}[label=(3\alph*)]
\item the set $\A := \bigcup_{\sigma\in\omega^{\omega}}\bigcap_{i\ge 1} A_{\sigma|i}$ is linearly dense in $X$ and countably supports $X^*$; moreover $A_\emptyset=\A$ and $A_t \subset \A$, for every $t \in \omega^{<\omega}$;
    \item\label{it:nicePRIwAnalytic} for every $x\in\A$ it holds that $\{Q_\alpha(x)\colon \alpha\leq \kappa\}\subset \{0,x\}$;
    \item\label{it: PRIAjshrinkwAnalytic} for every $\varepsilon>0$, $x^*\in X^*$, and $\sigma\in \omega^{\omega}$ there exists $i\in \omega$ such that $(Q_\alpha)_{\alpha\leq\kappa}$ is $(A_{\sigma|i},\varepsilon)$-shrinking in $x^*$.
\end{enumerate}
\end{enumerate}
\end{thmx}

It is worth mentioning that there is a rich theory concerning the class of weakly $\K$-analytic spaces, see e.g. \cite[Section 4.4]{FabRedBook} and references therein. From more recent contributions we should mention the paper \cite{AAM08} where the authors solved a long standing open problem by M. Talagrand by finding a weakly $\mathcal{K}$-analytic space $X$ which is not $F_{\sigma\delta}$ in $(X^{**},w^*)$. This result motivated even some very recent works, see e.g. \cite{KK18}.

Our characterization of Va\v{s}\'ak spaces is the following.

\begin{thmx}\label{thm:Vasak}
Let $X$ be a Banach space and put $\kappa:=\dens(X)$. Then the following conditions are equivalent.
\begin{enumerate}[label = (\roman*)]
    \item\label{it:Xvasak} $X$ is a Va\v{s}\'{a}k space.
    \item\label{it:PSshrinkvasak} There exist a projectional skeleton $\mathfrak{s}=(P_s)_{s\in\Gamma}$ on $X$ and a family $\{A_n \subset B_X\colon n \in \omega\}$ of non-empty sets satisfying the following conditions:
\begin{enumerate}
    \item\label{it:lindenscountsupportvasak}  $\A := \bigcup_{n\in\omega} A_n$ is linearly dense in $X$ and countably supports $X^*$;
    \item\label{it:epsilon-x-x*}for every $\varepsilon >0$ and every $x^*\in X^*$, there exists $N\subset \omega$ such that $\A=\bigcup_{j\in N}A_j$ and $\mathfrak{s}$ is $(A_j,\varepsilon)$-shrinking in $x^*$, for every $j\in N$.
\end{enumerate}
\item\label{it:SPRIvasak} There exist a SPRI $(Q_{\alpha})_{\alpha\leq \kappa}$ in $X$ and a family $\{A_n\subset B_X \colon n \in \omega\}$ of non-empty sets satisfying the following conditions:
\begin{enumerate}[label=(3\alph*)]
\item $\A := \bigcup_{n\in\omega} A_n$ is linearly dense in $X$ and countably supports $X^*$;
\item\label{it:nicePRIvasak} for every $x\in\A$ it holds that $\{Q_\alpha(x)\colon \alpha\leq \kappa\}\subset \{0,x\}$;
\item\label{it: PRIAjshrinkvasak} for every $\varepsilon>0$ and $x^*\in X^*$, there exists $N\subset \omega$ such that $\A=\bigcup_{n\in N}A_{n}$ and $(Q_\alpha)_{\alpha\leq\kappa}$ is $(A_j,\varepsilon)$-shrinking in $x^*$, for every $j\in N$.
\end{enumerate}
\end{enumerate}
\end{thmx}

Note that Va\v{s}\'ak spaces are known also under the name \emph{weakly countably determined spaces}. We refer the interested reader to \cite{KalVasak} for more information and recent contributions to the study of this class of spaces.

Our characterizations were inspired by the papers \cite{FGMZ} and \cite{FM18}. In \cite{FGMZ}, characterizations of these classes were found but the notion of a skeleton is not used, since this notion was discovered only years later and ``shrinkingness'' is replaced with a certain combinatorial property. This probably inspired M. Fabian and V. Montesinos \cite{FM18} to characterize WCG and SWCG spaces using the notion of a projectional skeleton, where the combinatorial conditions from \cite{FGMZ} were replaced by the notion of ``shrinkingness''.
The hard part in the proofs of the characterizations presented in \cite{FM18} was to show that a certain property of a projectional skeleton is inherited by the PRI that comes from this skeleton; then a transfinite induction argument was used. However, when dealing with Va\v{s}\'ak spaces and weakly $\mathcal{K}$-analytic spaces, even after several attempts we were not able to use this approach (due to various technical reasons) and we were forced to work differently. Namely, we prove that a certain property of a skeleton is actually inherited also by the SPRI that comes from this skeleton, which is essential in what is contained in the proof of the implication \ref{it:PSshrinkwanalytic}$\Rightarrow$\ref{it:SPRIwanalytic} in Theorems~\ref{thm:main} and ~\ref{thm:Vasak}. As a consequence, we do not only have characterizations of Va\v s\'ak spaces and weakly $\mathcal K$-analytic spaces using the notion of a projectional skeleton, but also using SPRI, which seem to be of independent interest. The techniques developed enable also to provide similar characterizations of WCG and SWCG spaces using the notion of SPRI, which are new as well.

We would like to stress out that condition \ref{it:PSshrinkwanalytic} from Theorems~\ref{thm:main} and \ref{thm:Vasak} is not that much dependent on a particular projectional skeleton. This is witnessed by the following consequence of Theorem~\ref{thm:subskeletons}, which seems to be of an independent interest and which might be considered as one of the main outcomes of this paper as well.

\begin{thmx}\label{thm:main1}
Let $X$ be a Banach space and $\{\mathfrak{s}_n\colon n \in \omega\}$ be a countable family of projectional skeletons on $X$ inducing the same set $D\subset X^*$. Then there exists a simple projectional skeleton on $X$ which is isomorphic to a subskeleton of $\mathfrak{s}_n$, for every $n\in\omega$.
\end{thmx}

Note that Theorem~\ref{thm:main1} implies that some properties of projectional skeletons are not dependent on the particular skeleton but rather on the set they induce (the fact that it is even simple is just an additional feature related to the concept studied already in \cite{CK14}). A result of similar purpose was proved already in \cite[Lemma 3.5]{kal20}, but our Theorem~\ref{thm:main1} is stronger and easier for applications. It also generalizes \cite[Theorem 4.1]{CK14}, as it shows that up to passing to a subskeleton, we may assume that any projectional skeleton is \emph{simple}, that is, that it can be indexed by the ranges of the corresponding projections. We refer the reader to Remark~\ref{ourresult-generalizes} for some more comments concerning the novelty of this result.

Now we describe the content of each section. Section~\ref{sec:prelim} contains basic notations and some preliminary results.
In Section~\ref{sec:models} we deal with the canonical construction of projections induced by projectional skeletons. To this aim, we use the method of suitable models. The main outcome here is the proof of Theorem~\ref{thm:main1}. Some technical constructions used in Section~\ref{sec:mainResults} are proved as well; they are concentrated in Subsection~\ref{subsec:usedElsewhere}.
In Section~\ref{sec:mainResults} we prove our main results, that is, characterizations of WCG, SWCG, weakly $\mathcal{K}$-analytic and Va\v{s}\'ak spaces using the notions of projectional skeleton and SPRI. In Section~\ref{sec:questions} we suggest directions for a possible further research.

\section{Notation and Preliminaries}\label{sec:prelim}

Given a set $A$ by $\PP(A)$ we denote the collection of its subsets.

We use standard notation in Banach space theory as can be found in \cite{FHHMZ}. By a projection on a Banach space $X$ we always understand a linear bounded mapping $P\colon X\rightarrow X$ such that $P\circ P=P$. Let us recall the notion of a projectional skeleton on a Banach space. Let $(\Gamma,\leq)$ be an up-directed partially ordered set. We say that a sequence $(s_n)_{n\in\omega}$ of elements of $\Gamma$ is increasing if $s_n\leq s_{n+1}$, for every $n\in\omega$. We say that  $\Gamma$ is \emph{$\sigma$-complete} if for every increasing sequence $(s_n)_{n\in\omega}$ in $\Gamma$ there exists $\sup_{n \in \omega} s_n$ in $\Gamma$. If $\Gamma$ is $\sigma$-complete, then we say that a subset $\Gamma'$ of $\Gamma$ is \emph{$\sigma$-closed} in $\Gamma$ if for every increasing sequence $(s_n)_{n\in\omega}$ in $\Gamma'$, it holds that $\sup_{n \in \omega} s_n \in \Gamma'$; finally, given $A\subset \Gamma$ we denote by $A_\sigma$ the smallest $\sigma$-closed subset of $\Gamma$ containing $A$.

\begin{defin}\label{def:C-ProjectionalSkeleton}
Let $X$ be a Banach space. A \emph{projectional skeleton} on $X$ is a family of bounded linear projections $\mathfrak{s}=(P_s)_{s \in \Gamma}$ on $X$ indexed by a partially ordered, up-directed, and $\sigma$-complete set $\Gamma$ satisfying:
\begin{enumerate}
    \item\label{it:separable}$P_s[X]$ is separable, for every $s \in \Gamma$;
    \item\label{it:image-increasing} if $s\le t$, then $P_s=P_s \circ P_t=P_t \circ P_s$;
    \item\label{it:continuity-skeleton} if $(s_n)_{n \in \omega}$ is an increasing sequence in $\Gamma$ and $s:=\sup_{n \in \omega} s_n$, then $P_s[X]=\overline{\bigcup_{n \in \omega}P_{s_n}[X]}$;
    \item\label{it:covers-X} $X=\bigcup_{s \in \Gamma} P_s[X]$.
\end{enumerate}
We put $D(\mathfrak{s}):=\bigcup_{s \in \Gamma}P_s^*[X^*]$ and we say that $D(\mathfrak{s})$ is the \emph{set induced by $\mathfrak{s}$}.
\end{defin}

We introduce also the following definition, which we use further as well.

\begin{defin}\label{def:subskeleton-isomorphic}
    Let $X$ be a Banach space and let $\mathfrak{s} = (P_s)_{s\in \Gamma}$ be a projectional skeleton on $X$.
    \begin{itemize}
    \item We say that $\mathfrak{s}':=(P_s)_{s \in \Gamma'}$ is a \emph{projectional subskeleton} of $\mathfrak{s}$ if $\Gamma'$ is a $\sigma$-closed subset of $\Gamma$ and $\mathfrak{s}'$ is a projectional skeleton on $X$.
    \item We say that a projectional skeleton $(Q_\lambda)_{\lambda \in \Lambda}$ on $X$ is \emph{isomorphic} to $\mathfrak{s}$ if there exists an order-isomorphism $\phi\colon\Lambda \to \Gamma$ such that $Q_\lambda=P_{\phi(\lambda)}$, for every $\lambda \in \Lambda$.\\
    (Recall that by an order-isomorphism we understand a bijection preserving the ordering, that is, $\phi(s)\leq \phi(t)$ iff $s\leq t$)
\end{itemize}

\end{defin}

Note that, given $\mathfrak{s}'$ a subskeleton of $\mathfrak{s}$, we have $D(\mathfrak{s}')=D(\mathfrak{s})$. Indeed, it is clear that $D(\mathfrak{s}') \subset D(\mathfrak{s})$ and thus it follows from \cite[Corollary~19]{kubis09} that $D(\mathfrak{s}')=D(\mathfrak{s})$.

It follows from \cite[Proposition~9 and Lemma~10]{kubis09} that given a projectional skeleton $\mathfrak{s}=(P_s)_{s \in \Gamma}$ on a Banach space $X$, up to passing to a subskeleton, we can assume that there exists $C \ge 1$ such that $\Vert P_s \Vert \le C$, for every $s \in \Gamma$ and the following stronger version of condition \eqref{it:continuity-skeleton} holds:
\begin{enumerate}
    \item[(3')] If $(s_n)_{n \in \omega}$ is an increasing sequence in $\Gamma$ and $s=\sup_{n \in \omega}s_n$, then $\lim_{n \to \infty}P_{s_n}(x)=P_s(x)$, for every $x \in X$.
\end{enumerate}

In the next lemma we recall one aspect of the strong relationship between projectional skeletons and norming subspaces. Given a Banach space $X$ and a real number $C \ge 1$, a subset $D$ of $X^*$ is said to be \emph{$C$-norming} if for every $x \in X$ it holds that
\[\Vert x \Vert \le C \sup\{ \vert d(x)\vert/\Vert d\Vert\colon d \in D \setminus \{0\}\}.\]

\begin{lemma}[{\cite[Lemma~1.3]{kal20}}]\label{D-norming}
If $\mathfrak{s}$ is a projectional skeleton on a Banach space $X$, then there exists $C \ge 1$ such that $D(\mathfrak{s})$ is a $C$-norming subspace of $X^*$.
\end{lemma}

In the statement of Theorem~\ref{thm:main1}, following \cite{CK14}, we use the notion of a simple projectional skeleton. Let us recall it.

\begin{defin}[{J.M. Borwein and W. Moors \cite{BM}}]
Let $X$ be a Banach space and $\mathcal{F}$ be a family of closed separable subspaces of $X$. We say that $\mathcal{F}$ is \emph{rich} if
\begin{enumerate}
    \item each closed separable subspace of $X$ is contained in an element of $\mathcal{F}$ and
    \item for every increasing sequence $(F_n)_{n\in \omega}$ in $\mathcal{F}$, we have $\overline{\bigcup_{n\in\omega} F_n}\in\F$.
\end{enumerate}
\end{defin}
\begin{defin}
A \emph{simple projectional skeleton} on a Banach space $X$ is a family of projections $(P_F)_{F\in\mathcal{F}}$ on $X$ indexed by a rich family $\mathcal{F}$ in $X$ ordered by inclusion such that
\begin{enumerate}
    \item for every $F\in\mathcal{F}$ we have $P_F(X) = F$ and
    \item if $E\subset F$ in $\mathcal{F}$, then $P_E = P_E\circ P_F = P_F \circ P_E$.
\end{enumerate}
\end{defin}

Clearly, every simple projectional skeleton is a projectional skeleton. We also refer the interested reader to \cite{FKK}, where several natural variants of the concept of a simple projectional skeleton were studied.

In our main results we use the notions of SPRI and of shrinkingness. The notion of SPRI is nowadays classical, see e.g. \cite[Definition~6.2.6]{FabRedBook}. The notion of shrinkingness as we introduce it below is new. It is inspired by related concepts of shrinkingness introduced in \cite{FM18}.

\begin{defin}
Let $X$ be a nonseparable Banach space. We say that a family of bounded linear projections $(Q_\alpha)_{\alpha\in[0,\dens X]}$ is a \emph{separable projectional resolution of the identity (SPRI)} in $X$, if $Q_0 = 0$, $Q_{\dens X} = \Id$ and for every $\alpha\in [0,\dens X)$ the following holds
\begin{enumerate}
    \item $(Q_{\alpha+1}-Q_\alpha)[X]$ is separable,
    \item $Q_\alpha Q_\beta = Q_\beta Q_\alpha = Q_\beta$ for $\beta\in[0,\alpha]$, and
    \item $x\in \closedSpan\{(Q_{\beta+1} - Q_\beta)x\colon 0\leq \beta < \dens X\}$, for every $x\in X$.
\end{enumerate}
\end{defin}

Let $X$ be a Banach space and $A\subset X$ be a non-empty bounded set. We define $\rho_A\colon X^*\times X^*\to \mathbb{R}$ as
\[
\rho_A(x^*,y^*):=\sup_{x\in A}|x^*(x)-y^*(x)|, \quad (x^*,y^*) \in X^* \times X^*.
\]
It is clear that $\rho_A$ is a pseudo-metric on $X^*$.

\begin{defin}
Let $X$ be a Banach space, $\Gamma$ be a partially ordered set, $\mathfrak{s} = (P_s\colon s \in \Gamma)$ be a family of bounded projections on $X$, and $A$ be a bounded subset of $X$. Given $\varepsilon\geq 0$ and $x^* \in X^*$, we say that $\mathfrak{s}$ is \textit{$(A, \varepsilon)$-shrinking in $x^*$} if for every increasing sequence $(s_n)_{n \in \omega}$ of elements of $\Gamma$ such that $s=\sup_{n \in \omega}s_n$ exists in $\Gamma$, it holds that $\limsup_{n \to \infty} \rho_A(P_{s_n}^*(x^*), P_s^*(x^*)) \le \varepsilon \Vert x^* \Vert$. We also put:
\[\mathcal{T}(\mathfrak{s}, A):=\{(\varepsilon, B, x^*) \colon \varepsilon\geq 0, B \subset A, x^* \in X^* \ \text{and $\mathfrak{s}$ is $(B, \varepsilon)$-shrinking in $x^*$}\}.\]
\end{defin}

Finally, we collect some notions related to the class of Plichko spaces. Given a subset $A$ of a Banach space $X$ and $x^* \in X^*$, we say that $A$ \emph{countably supports $x^*$} if $\suppt_{A}(x^*):=\{x \in A\colon x^*(x) \ne 0\}$ is countable, and we say that $A$ \emph{countably supports} a subset $D$ of $X^*$ if $A$ countably supports every element of $D$. We recall that a Banach space $X$ is \emph{Plichko} if there exist $C\ge1$, a $C$-norming set $D\subset X^*$ and a linearly dense set $A\subset X$ such that $A$ countably supports $D$. It is nowadays well-known that a Banach space admits a commutative projectional skeleton if and only if it is a Plichko space, see \cite[Theorem~27]{kubis09} for a proof using suitable models and \cite[Theorem~3.1]{kal20} for a proof without using suitable models. Recall that a projectional skeleton $(P_s)_{s \in \Gamma}$ is said to be \emph{commutative} if $P_s \circ P_t=P_t \circ P_s$, for every $s, t \in \Gamma$.

\section{Projections associated to suitable models}\label{sec:models}

In this Section we try to convince the reader that the natural approach when studying Banach spaces with a projectional skeleton is to use the method of suitable models. The results which are needed in further Sections are concentrated in the last subsection.

\subsection{Preliminaries concerning suitable models}

Here we settle the notation and give some basic observations concerning suitable models. We refer the interested reader to \cite{C12}, where more details about this method may be found (warning: in \cite{C12} only \textbf{countable} models were considered, while here we consider suitable models which are not necessarily countable).

Let $M$ be a fixed set and $\phi$ be a formula in the language of ZFC. Then the {\em relativization of $\phi$ to $M$} is the formula, denoted by $\phi^M$, which is obtained from $\phi$ by replacing each quantifier of the form ``$\exists x$'' by ``$\exists x\in M$'' and each quantifier of the form ``$\forall x$'' by ``$\forall x\in M$''. If $\phi(x_1,\ldots,x_n)$ is a formula with all free variables shown (i.e. a formula whose free variables are exactly $x_1,\ldots, x_n$), then we say that {\em $\phi$ is absolute for $M$} if
\[
\forall a_1,\ldots,a_n\in M\quad (\phi^M(a_1,\ldots,a_n) \leftrightarrow \phi(a_1,\ldots,a_n)).
\]

\begin{defin}
\rm Let $\Phi$ be a finite list of formulas and let $X$ and $M$ be sets. We say that $M$ is a \emph{suitable model for $\Phi$ and $X$} and we write $M \prec (\Phi; X)$, if $X \subset M$ and every formula of $\Phi$ is absolute for $M$.
\end{defin}

Note that the Reflexion Theorem \cite[Theorem~IV.7.4]{kunenBook} ensures that given a set $X$ and a finite list of formulas $\Phi$, there exists a set $M$ such that $M \prec (\Phi; X)$. Actually, we have the following stronger result that follows from \cite[Theorem IV.7.8]{kunenBook}.

\begin{thm}\label{thm:modelExists}
Let $\Phi$ be a finite list of formulas and $X$ be any set. Then there exists a set $R$ such that $R\prec (\Phi; X)$, $|R|\leq \max(\omega,|X|))$ and moreover, for every countable set $Z\subset R$ there exists a countable set $M\subset R$ such that $M\prec(\Phi;\; Z)$.
\end{thm}

The fact that a certain formula is absolute for $M$ will always be used exclusively in order to satisfy the assumption of the following lemma. Using this lemma we can force the model $M$ to contain all the needed objects created (uniquely) from elements of $M$.

\begin{lemma}\label{l:unique-M}
Let $\phi(y,x_1,\ldots,x_n)$ be a formula with all free variables shown and let $M$ be a set such that $\phi$ and  $\exists y \phi(y,x_1, \ldots, x_n)$ are absolute for $M$. If $a_1, \ldots, a_n\in M$ are such that there exists a set $u$ satisfying $\phi(u, a_1, \ldots, a_n)$, then there exists a set $v \in M$ satisfying $\phi(v, a_1, \ldots, a_n)$. Moreover, if there exists a unique set $u$ such that $\phi(u, a_1, \ldots, a_n)$, then $u \in M$.
\end{lemma}
\begin{proof}
See \cite[Lemma~5]{CCS}.
\end{proof}

\begin{convention}
Whenever we say ``\emph{for any suitable model $M$ (the following holds \dots)}''
we mean that  ``\emph{there exists a finite list of formulas $\Phi$ and a countable set $S$ such that for every $M \prec (\Phi;S)$ (the following holds \dots)}''.
\end{convention}

Given a topological space $X$ and a set $M$, we put
\[X_M:=\overline{X\cap M}.\]
If $M$ is a set and $\langle X,\tau\rangle$ is a topological space or $\langle X,d\rangle$ is a metric space or $\langle X,+,\cdot,\|\cdot\|\rangle$ is a normed linear space, then we say that \emph{$M$ contains $X$} if $\langle X,\tau\rangle\in M$, $\langle X,d\rangle\in M$ or $\langle X,+,\cdot,\|\cdot\|\rangle\in M$, respectively.

Given an up-directed and $\sigma$-complete partially ordered set $(\Gamma, \le)$ and a projectional skeleton $\mathfrak{s}=(P_s)_{s \in \Gamma}$ on a Banach space $X$, we say that a set $M$ \emph{contains} $\mathfrak{s}$ if $(\Gamma, \le) \in M$ and $P\in M$, where $P\colon\Gamma \to \mathcal{L}(X,X)$ is the mapping given by $P(s)=P_s$, for every $s \in \Gamma$.

Basic properties of suitable models were presented in \cite[Lemma~7 and Lemma~8]{CCS}. In the next lemma we summarize further properties that will be used in this work.

\begin{lemma}\label{l:basics}
For every suitable model $M$ the following holds:
\begin{enumerate}
\item\label{it:homeo} If $f\colon T\to W$ is a homeomorphism between topological spaces and $f \in M$, then $f[\overline{T\cap M}] = \overline{W\cap M}$.
\item\label{it:models-Banach} Let $X$ be a Banach space such that $M$ contains $X$.
\begin{enumerate}[label = (\roman*)]
    \item\label{it:hull} The operator $\Span$ which assigns to every subset of $X$ its linear hull belongs to $M$.

    \item\label{it:linearlyDenseSubset} If $A\subset X$ is linearly dense and $A\in M$, then $\closedSpan (A\cap M) = X_M$.

    \item\label{it: weakstar-closure-subspace} If $D$ is a linear subspace of $X^*$ and $M$ contains $D$, then $\overline{D\cap M}^{\,w^*}=((D\cap M)_\perp)^\perp$.
\end{enumerate}

\end{enumerate}
\end{lemma}
\begin{proof}
Let $S$ be the union of the countable sets from the statements of \cite[Lemma 7 and Lemma 8]{CCS} and let $\Phi$ be the union of the finite lists of formulas from the statements of \cite[Lemma 7 and Lemma 8]{CCS} enriched by the formulas (and their subformulas) marked by $(*)$ in the proof below. Let $M\prec (\Phi; S)$.

Let $f\colon T \to W$ be as in \eqref{it:homeo}. Given any $w \in W \cap M$, using \cite[Lemma 7(2)]{CCS}, Lemma \ref{l:unique-M} and the absoluteness of the following formula (and its subformulas)
\[
\exists t \in T\colon\quad w=f(t),\eqno{(*)}
\]
we conclude that there exists $t \in T \cap M$ such that $w=f(t)$. This implies that $W \cap M \subset f[T \cap M]$ and thus $\overline{W \cap M} \subset f[\overline{T \cap M}]$. Note that \cite[Lemma 8(3)]{CCS} ensures that $f^{-1} \in M$. Thus it follows from the proved above that $\overline{T \cap M} \subset f^{-1}[\overline{W \cap M}]$, which implies that $f[\overline{T \cap M}] \subset \overline{W \cap M}$.

Now, let $X$ and $M$ be as in \eqref{it:models-Banach}. Item \ref{it:hull} follows from Lemma~\ref{l:unique-M} and the absoluteness of the following formula (and its subformulas)
\[\begin{split}
\exists \Span\colon\PP(X)\to\PP(X)\colon\quad & (\forall A\subset X\colon \Span A\text{ is the minimal linear subspace}\\ & \text{ containing the set }A).\qquad (*)
\end{split}\]
In order to prove \ref{it:linearlyDenseSubset}, note that \cite[Lemma 7(7)]{CCS} ensures that $X_M$ is a (closed) linear subspace of $X$ and thus $\closedSpan (A\cap M) \subset X_M$. On the other hand, for every $x\in X\cap M$, using the fact that $A$ is linearly dense in $X$, item \ref{it:hull}, Lemma~\ref{l:unique-M} and the absoluteness of the following formula (and its subformulas)
\[
\exists S\subset A\colon\qquad (S\text{ is countable and }x\in \closedSpan S),\eqno{(*)}
\]
we conclude that there exists a countable set $S\subset A$ such that $S\in M$ and $x\in \closedSpan S$. It follows from \cite[Lemma 7(4)]{CCS} that $S\subset M$ and thus we obtain that $x\in \closedSpan (A\cap M)$. This shows that $X \cap M \subset  \closedSpan (A\cap M)$ and therefore $X_M\subset \closedSpan (A\cap M)$. Now let us prove \ref{it: weakstar-closure-subspace}. Note that since the linear operations of $D$ belong to $M$, using the fact that $\mathbb{Q} \subset M$ and \cite[Lemma 7(2)]{CCS}, we conclude that $D \cap M$ is a $\mathbb{Q}$-linear subspace of $X^*$ and therefore $\overline{D \cap M}^{\,w^*}$ is a linear subspace of $X^*$. Therefore, we have that
\[((D\cap M)_\perp)^\perp =\closedSpan^{\,w^*}(D\cap M)=\overline{D\cap M}^{\,w^*}.\qedhere\]
\end{proof}

In the rest of the paper we will use several times arguments similar to those presented in the proof of Lemma \ref{l:basics}. To simplify the presentation, we adopt the following terminology. Whenever we say that ``it follows from absoluteness that some object $\Omega$ is in $M$'' or that ``it follows from absoluteness that there exists an object $\Omega$ in $M$ with a certain property'', we mean that the existence of such $\Omega$ in $M$ can be established using Lemma \ref{l:unique-M} for an appropriate formula $\phi$ (which either uniquely defines $\Omega$ or states the property that $\Omega$ should satisfy) and that the formulas
$\phi$ and $\exists\Omega\,\phi$ should be added to the finite list of formulas that are absolute for $M$.

Next Lemma describes a very concrete construction of suitable models. The reason for including it here is twofold. First, it shows that there is nothing too abstract when working with suitable models and this terminology just enables us to replace complicated inductive constructions by an abstract notion. Second, it is one of the key tools in the proof of Theorem \ref{thm:subskeletons}. The construction is more-or-less standard. It is described also in \cite[Lemma 2.4]{CK14} (where it is formulated for countable models only); for the convenience of the reader we provide a full proof below.
\begin{lemma}\label{l:skolem}
Let $\Phi$ be a finite subformula closed list of formulas and $R$ be a set such that $R\prec (\Phi; \emptyset)$. Then there is a mapping $\psi\colon\PP(R)\to \PP(R)$ (called the \emph{Skolem function}) satisfying the following
\begin{enumerate}
    \item\label{it:isModel} \[\forall A\subset R\colon \quad \psi(A)\prec (\Phi; A) \text{ and } |\psi(A)| \le \max(\omega, |A|),\]
    \item\label{it:monotone} The mapping $\psi$ is monotone, i.e., $\psi(A)\subset \psi(B)$ whenever $A, B\in\PP(R)$ are such that $A\subset B$.
    \item\label{it:idempotent} The mapping $\psi$ is idempotent, i.e., $\psi\circ \psi = \psi$.
    \item\label{it:union} For every $\F\subset \PP(R)$ such that $\{\psi(F)\colon F\in\F\}$ is up-directed, we have  that $\psi(\bigcup\F) = \bigcup_{F\in\F}\psi(F)$.
    \item\label{it:construction} Let $J\subset R$ be an arbitrary set. Then, for every $A\subset J$ and $B\subset R$, we have that
	\[\psi(A)\cap J\subset \psi(B) \Longleftrightarrow \psi(A)\subset \psi(B).\]
	\item\label{it:skolem-intersection} For every $\F \subset \PP(R)$, we have that \[\psi\Big(\bigcap_{F \in \F}\psi(F)\Big) = \bigcap_{F\in \F}\psi(F).\]
\end{enumerate}
\end{lemma}
\begin{proof}
Fix a well-ordering $\triangleleft$ on the set $R$ and let $\varphi_1,\ldots,\varphi_n$ be the formulas from the list $\Phi$.
For every $i\in\{1,\ldots,n\}$, denote by $l_i$ the number of all the free variables of the formula $\varphi_i$ and consider the mapping $H_i\colon R^{l_i}\to R$ defined as follows:
\begin{itemize}
	\item if $l_i = 0$, then $R^{l_i} = \{\emptyset\}$ and $H_i(\emptyset)$ is the $\triangleleft$-least element of $R$.
	\item if $l_i > 0$ and $(r_1,\ldots,r_{l_i})\in R^{l_i}$ is fixed, then:
		\begin{itemize}
			\item[-] if there exists $j \in \{1,\ldots,n\}$ such that $\varphi_i = \exists x\,\,\varphi_j(x,y_1,\ldots,y_{l_i})$ and there exists $r\in R$ such that $\varphi_j(r,r_1,\ldots,r_{l_i})$ holds, then $H_i(r_1,\ldots,r_{l_i})$ is the $\triangleleft$-least of such elements.		
			\item[-] in all the other cases, $H_i(r_1,\ldots,r_{l_i})$ is the $\triangleleft$-least element of $R$.
		\end{itemize}
\end{itemize}
Fixed $A\in \PP(R)$, we define $\psi(A):=\bigcup_{k \in \omega} A_k$, where $(A_k)_{k \in \omega}$ is the increasing sequence of subsets of $R$ built by recursion such that $A_0:=A$ and
$$A_{k+1} := A_k\cup \bigcup\{H_i(a_1,\ldots,a_{l_i})\setsep i = 1,\ldots,n\;,(a_1,\ldots,a_{l_i})\in (A_k)^{l_i}\},$$
for every $k \in \omega$. Note that it follows immediately from the definition of the Skolem function $\psi$ that $|\psi(A)|\le \max(\omega, |A|)$ and that \eqref{it:monotone} and \eqref{it:idempotent} hold. Moreover using \cite[Lemma 2.1]{C12} and the fact that every formula of $\Phi$ is absolute for $R$, we conclude that $\psi(A)\prec (\Phi; A)$, since clearly $A \subset \psi(A)$.

Let $\F$ be as in \eqref{it:union}. Then it follows from \eqref{it:monotone} that $\psi(\bigcup\F) \supset \bigcup_{F\in\F}\psi(F)$. In order to prove the other inclusion, firstly we claim that $\psi(\bigcup_{F\in\F}\psi(F)) = \bigcup_{F\in\F}\psi(F)$. Indeed, let $(A_k)_{k \in \omega}$ be the sequence used in the definition of $\psi(\bigcup_{F\in\F}\psi(F))$ and let us show by induction that $A_k=\bigcup_{F\in\F}\psi(F)$, for every $k \in \omega$. It is clear that $A_0=\bigcup_{F\in\F}\psi(F)$. Now fix $k \in \omega$ and assume that $A_k=\bigcup_{F\in\F}\psi(F)$. Note that fixed $i\in\{1,\ldots,n\}$ and $a_1,\ldots,a_{l_i}\in A_k=\bigcup_{F\in\F}\psi(F)$, there exists $F\in\F$ such that $a_1,\ldots,a_{l_i}\in\psi(F)$, since $\{\psi(F)\colon F\in\F\}$ is up-directed. This implies that $H_i(a_1,\ldots,a_{l_i})\in \psi(F)$ and thus  $A_{k+1}= \bigcup_{F\in\F}\psi(F)$. Since $\psi(\bigcup_{F\in\F}\psi(F))=\bigcup_{k \in \omega} A_k$, we conclude the claim. Finally, since $\bigcup\F\subset \bigcup_{F\in\F}\psi(F)$, using \eqref{it:monotone} we obtain that $\psi(\bigcup\F)\subset \psi(\bigcup_{F\in\F}\psi(F))=\bigcup_{F\in\F}\psi(F)$.

Let $J,A,B$ be as in \eqref{it:construction}. It is obvious that if $\psi(A) \subset \psi(B)$, then $\psi(A) \cap J \subset \psi(B)$. Now suppose that $\psi(A)\cap J\subset \psi(B)$ and note that this implies that $A \subset \psi(B)$. Indeed, using the fact that $A \subset \psi(A)$, we obtain that:
\[A=A \cap J \subset \psi(A) \cap J \subset \psi(B).\]
Therefore, using \eqref{it:monotone} and \eqref{it:idempotent}, we conclude that $\psi(A) \subset \psi(B)$.

Finally, fix $\F \subset \PP(R)$ and let $(A_k)_{k \in \omega}$ be the sequence used in the definition of $\psi\Big(\bigcap_{F \in \F}\psi(F)\Big)$. By an inductive argument similar to the presented in the proof of \eqref{it:union}, we conclude that $A_k=\bigcap_{F \in \F}\psi(F)$, for every $k \in \omega$. This establishes \eqref{it:skolem-intersection}, since $\psi\Big(\bigcap_{F \in \F}\psi(F)\Big)=\bigcup_{k \in \omega} A_k$.
\end{proof}

The following proposition goes essentially back to the proof of \cite[Theorem 15]{C18}. We record it here for further use in Subsection~\ref{subsec:usedElsewhere}. As usual we denote the density of a topological space $X$ by $\dens X$.

\begin{prop}\label{prop:inclusions}There exist a countable set $S$ and a finite list of formulas $\Phi$ such that the following holds: Let $X$ be a topological space which is homeomorphic to a Banach space and put $\kappa := \dens X$. Then there exists a countable set $S'\supset S$ such that
\[
\forall M,N\prec(\Phi;S')\colon\quad M\cap \kappa\subset N\cap \kappa \Leftrightarrow X_M\subset X_N.
\]
\end{prop}
\begin{proof}
Let $S$ be the union of the countable sets from the statements of Lemma~\ref{l:basics} and \cite[Lemma~7 and Lemma~8]{CCS} and let $\Phi$ be the union of the finite lists of formulas from the statements of Lemma~\ref{l:basics} and \cite[Lemma~7 and Lemma~8]{CCS}.
By the result of H. Toru\'{n}czyk \cite{Tor81}, all infinite-dimensional Banach spaces with the same density are topologically homeomorphic. Thus, there exists a (not necessarily linear) homeomorphism $f\colon\ell_2(\kappa)\to X$. Let $(e_i)_{i\in \kappa}$ be the canonical orthonormal basis of $\ell_2(\kappa)$, let $e\colon\kappa\to\ell_2(\kappa)$ be the map given by $e(i):=e_i$, for every $i<\kappa$ and let $S'$ be a countable set that contains $X$, $\ell_2(\kappa)$ and such that $S \cup \{f,e\} \subset S'$.

Note that $X_M = f\big[\closedSpan \{e_i\colon i\in \kappa\cap M\}\big]$, for every $M\prec (\Phi;S')$. Indeed, it follows from Lemma~\ref{l:basics}\eqref{it:models-Banach}\ref{it:linearlyDenseSubset} and \cite[Lemma~7(2)]{CCS} that $\overline{\ell_2(\kappa) \cap M}=\closedSpan \big(\{e_i\colon i \in \kappa\} \cap M\big)$. Thus using Lemma~\ref{l:basics}\eqref{it:homeo} and \cite[Lemma~8(4)]{CCS}, we conclude that $X_M = f\big[\closedSpan \{e_i\colon i\in \kappa\cap M\}\big]$. Now fix any $M,N\prec(\Phi;S')$. By the above, it is easy to observe that $M\cap \kappa\subset N\cap \kappa$ implies $X_M\subset X_N$. On the other hand, if $X_M\subset X_N$ and $i\in M\cap \kappa$ is given, then $f(e_i)\in X_M \subset X_N = f\big[\closedSpan\{e_j\colon j\in N\cap \kappa\}\big]$ and so we have that $i\in N\cap \kappa$, since $f$ is injective and $e_i\notin \closedSpan\{e_j\colon j\neq i, j<\kappa\}$.
\end{proof}

\subsection{Canonical projections in spaces with a projectional skeleton}

The next lemma is a key tool for the construction of projections in Banach spaces with a projectional skeleton. This is more-or-less known to the experts, see \cite[Lemma 4]{kubis09} and \cite[Lemma 5.1]{BHK17}, where proofs are essentially given for suitable models which are countable. Since this is a crucial tool for us and we need it for uncountable models as well, for the convenience of the reader we include a full proof here.

\begin{lemma}\label{l:projection}
For every suitable model $M$ the following holds: Let $X$ be a Banach space, $C\geq 1$ and $D$ be a $C$-norming subspace of $X^*$ such that $M$ contains $X$ and $D \in M$. Then \begin{enumerate}[label = (\roman*)]
    \item\label{it:intersection-XM} $X_M\cap (D\cap M)_\perp = \{0\}$,
    \item\label{it:projection} $X_M + (D\cap M)_\perp$ is a closed subspace of $X$ and the
    mapping
    \[
    X_M + (D\cap M)_\perp \ni (u+v)\longmapsto u\in X_M
    \]
    is a linear projection of norm at most $C$,
\end{enumerate}
and the following conditions are equivalent:
\begin{enumerate}[label = (\arabic*)]
\item\label{it:rangeIsX} $X_M + (D\cap M)_\perp = X$,
\item\label{it:separation} $X\cap M$ separates the points of $\overline{D\cap M}^{\,w^*}$,
\item\label{it:uniqueProjeciton} there exists a unique projection $P_M:X\to X$ with $P_M[X] = X_M$ and $\ker P_M = (D\cap M)_\perp$,
\item\label{it:uniqueProjeciton2} there exists a unique projection $P_M\colon X\to X$ with $P_M[X] = X_M$ and $d = d\circ P_M$ for every $d\in D\cap M$.
\end{enumerate}
Moreover, for any projection $Q\colon X\to X$ with $Q[X] = X_M$ we have $\ker Q = (D\cap M)_\perp$ if and only if $d = d\circ Q$ for every $d\in D\cap M$.
\end{lemma}
\begin{proof}
Let $S$ be the union of the countable sets from the statements of Lemma \ref{l:basics} and \cite[Lemma 7]{CCS} and let $\Phi$ be the union of the finite lists of formulas from the statements of Lemma \ref{l:basics} and \cite[Lemma 7]{CCS}, enriched by the finitely many formulas used in the proof below. Let $M\prec (\Phi; S)$ be such that $D \in M$ and $M$ contains $X$.

We claim that for every $x\in X_M$ and every $y\in (D\cap M)_\perp$, it holds that $\|x\|\leq C\|x+y\|$. Indeed, fix $x\in X\cap M$ and a rational number $q>C$. By Lemma~\ref{l:unique-M} and the absoluteness of the following formula (and its subformulas)
\[
\exists d \in D\colon \qquad \|x\| < q\tfrac{|d(x)|}{\|d\|},
\]
there is $d\in D\cap M$ such that $\|x\| < q\tfrac{|d(x)|}{\|d\|}$. Thus for every $y\in (D\cap M)_\perp$ we have that
\[
\|x\|< q\tfrac{|d(x)|}{\|d\|} = q\tfrac{|d(x+y)|}{\|d\|}\leq q\|x+y\|.
\]
Since $q\in (C,\infty)\cap \qe$ was arbitrary, we obtain that $\|x\|\leq C\|x+y\|$, for every $x\in X \cap M$ and $y\in (D\cap M)_\perp$, which implies the claim. It is easy to see that \ref{it:intersection-XM} and \ref{it:projection} follow from this claim.

\smallskip

\noindent\ref{it:rangeIsX}$\Rightarrow$\ref{it:separation}: If \ref{it:rangeIsX} holds then for every $x^*\in \overline{D\cap M}^{\, w^*}$ with $x^* \ne 0$, there exists $x\in X_M$ and $y\in (D\cap M)_\perp$ such that $x^*(x+y)\neq 0$. Note that Lemma \ref{l:basics}\eqref{it:models-Banach}\ref{it: weakstar-closure-subspace} ensures that $\overline{D\cap M}^{\,w^*}=((D\cap M)_\perp)^\perp$. Therefore, we have that $x^*(x)=x^*(x+y) \ne 0$, which implies that $X_M$ separates the points of $\overline{D\cap M}^{\,w^*}$ and thus \ref{it:separation} follows from the density of $X \cap M$ in $X_M$.

\smallskip

\noindent\ref{it:separation}$\Rightarrow$\ref{it:rangeIsX}: If \ref{it:separation} holds, then, using Lemma~\ref{l:basics} \eqref{it:models-Banach}\ref{it: weakstar-closure-subspace}, we have
\[(X_M + (D\cap M)_\perp)^\perp \subset (X_M)^\perp \cap ((D\cap M)_\perp)^\perp = (X_M)^\perp\cap \overline{D\cap M}^{\,w^*} = \{0\},\]
so we obtain that $X_M + (D\cap M)_\perp \supset \{0\}_\perp = X$ and hence \ref{it:rangeIsX} holds.

\smallskip

The equivalence between \ref{it:rangeIsX} and \ref{it:uniqueProjeciton} is obvious.
\smallskip

Finally, in order to conclude the proof, let us show that if $Q\colon X\to X$ is a projection with $Q[X] = X_M$, then $\ker Q=(D \cap M)_\perp$, if and only if $d=d \circ Q$, for every $d \in D \cap M$. It is easy to see that if $\ker Q=(D \cap M)_\perp$, then $d=d \circ Q$, for every $d \in D \cap M$ and the converse follows from \ref{it:intersection-XM}.
\end{proof}

It was proved in \cite[Theorem 15]{kubis09} that a Banach space $X$ admits a projectional skeleton if and only if there exist $C\geq 1$ and a $C$-norming subspace $D$ of $X^*$ such that  $X = X_M + (D\cap M)_\perp$ holds for every {\bf countable} suitable model $M$. It follows from Proposition \ref{p:characterizationPSkeletonNotCountableModels} below that we do not need to assume countability of the model and moreover Proposition \ref{p:characterizationPSkeletonNotCountableModels} provides us a formula for the canonical projection from Lemma~\ref{l:projection}, when $D$ is a set induced by a projectional skeleton.

\begin{prop}\label{p:characterizationPSkeletonNotCountableModels}
For every suitable model $M$, the following holds: Let $X$ be a Banach space and $\mathfrak{s}=(P_s)_{s \in \Gamma}$ be a projectional skeleton on $X$. If $M$ contains $X$ and $\mathfrak{s}$, then $\Gamma\cap M$ is up-directed and the mapping $P_M\colon X \to X$ given by
\begin{equation}
    P_M(x):=\lim_{s \in \Gamma \cap M}P_s(x), \quad x \in X,
\end{equation}
is a well-defined bounded projection such that $P_M[X]=X_M$ and $\ker P_M=(D(\mathfrak{s})\cap M)_\perp$.
\end{prop}
\begin{proof}
Let $S$ be the union of the countable sets from the statements of \cite[Lemma 7 and Lemma 8]{CCS} and Lemma~\ref{l:projection} and let $\Phi$ be the union of the finite lists of formulas from the statements of \cite[Lemma 7 and Lemma 8]{CCS} and Lemma~\ref{l:projection} enriched by the finitely many formulas needed in the proof below. Let $M\prec (\Phi; S)$ containing $X$ and $\mathfrak{s}$. Note that it follows from \cite[Lemma~8(1)]{CCS} that $\Gamma \cap M$ is up-directed and so \cite[Lemma~11]{kubis09} ensures that $P_M$ is a well-defined bounded projection on $X$ with $P_M[X]=\overline{\bigcup_{s \in \Gamma \cap M}P_s[X]}$.

First, we \emph{claim} that $P_M[X] = X_M$. Indeed, given any $s\in \Gamma\cap M$, since $P_s[X]$ is separable, by the absoluteness of the following formula (and its subformulas)
\[\exists C\subset X\colon\quad (C\text{ is countable and }\overline{P_s[C]} = P_s[X]),\]
we conclude that there exists a countable set $C\in M$ such that $\overline{P_s[C]} = P_s[X]$. Using \cite[Lemma 7 and 8]{CCS} we have that $C\subset M$ and $P_s[C]\subset X\cap M$, which implies that $P_s[X]\subset X_M$. And since $s\in \Gamma\cap M$ was arbitrary, we obtain $P_M[X]\subset X_M$. On the other hand, given $x\in X\cap M$, it follows from absoluteness that there exists $s\in \Gamma\cap M$ such that $x\in P_s[X]\subset P_M[X]$; hence, we also have that $X_M\subset P_M[X]$. This proves the claim.

Note that since $M$ contains $\mathfrak{s}$, it follows from absoluteness that $D(\mathfrak{s}) \in M$ and thus by Lemmas \ref{D-norming} and \ref{l:projection}, in order to conclude that $\ker P_M = (D(\mathfrak{s})\cap M)_\perp$, it suffices to show that $d = d\circ P_M$, for every $d\in D(\mathfrak{s})\cap M$. Given any $d\in D(\mathfrak{s})\cap M$, there is $s\in\Gamma$ with $d = P_s^*(d)$ and so it follows from absoluteness that there exists $s\in \Gamma\cap M$ such that $d = P_s^*(d) = d\circ P_s$, which implies that $d = d\circ P_M$.
\end{proof}

We believe that the projections from Lemma~\ref{l:projection} that were constructed in Proposition~\ref{p:characterizationPSkeletonNotCountableModels} are the key ingredients to handle inductive arguments in spaces with a projectional skeleton. Let us give it a name.

\begin{defin}\label{def:canonicalProjection}
Let $X$ be a Banach space, $C\geq 1$ and $D$ be a $C$-norming subspace of $X^*$. Given a set $M$, we say that $P_M$ is the \emph{canonical projection associated to $M$, $X$ and $D$} if it is the unique projection on $X$ satisfying $P_M[X]=X_M$ and $\ker P_M = (D\cap M)_\perp$. We say that a set $M$ \emph{admits canonical projection associated to $X$ and $D$} if there exists the canonical projection associated to $M$, $X$ and $D$.
\end{defin}

In Plichko spaces we may apply also the following lemma which we record here for further use.

\begin{lemma}\label{lem:countablysupports}
Let $X$ be a Banach space, $\mathfrak{s}$ be a projectional skeleton on $X$ and $A$ be a subset of $X$ that countably supports $D(\mathfrak{s})$. Then there exist a countable set $S$ and a finite list of formulas $\Phi$ such that every $M \prec (\Phi; S \cup \{A\})$, which contains $X$ and $\mathfrak{s}$, admits canonical projection $P_M$ associated to $X$ and $D(\mathfrak{s})$ and it holds that $P_M(x) \in \{0,x\}$, for every $x \in A$.
\end{lemma}
\begin{proof}
Let $S$ be the union of the countable sets from the statements of Proposition~\ref{p:characterizationPSkeletonNotCountableModels} and \cite[Lemma~7]{CCS} and let $\Phi$ be the union of the finite lists of formulas from the statements of Proposition~\ref{p:characterizationPSkeletonNotCountableModels} and \cite[Lemma~7]{CCS} enriched by the finitely many formulas used in the proof below .
Fix any $M \prec (\Phi; S \cup \{A\})$ such that $M$ contains $X$ and $\mathfrak{s}$. Consider $P_M$ from Proposition~\ref{p:characterizationPSkeletonNotCountableModels} (associated to $X$ and $D(\mathfrak{s})$). It is clear that $A \cap M \subset P_M[X]$.
Given any $x^* \in D(\mathfrak{s})\cap M$, by absoluteness we have that $\suppt_{A}(x^*) \in M$ and thus it follows from the countability of $\suppt_{A}(x^*)$ and \cite[Lemma~7(4)]{CCS} that $\suppt_{A}(x^*) \subset M$. This implies that $A \setminus M \subset (D(\mathfrak{s})\cap M)_\perp = \ker P_M$ and concludes the proof.
\end{proof}

\subsection{Proof of Theorem~\ref{thm:main1} and of other results used in further sections}\label{subsec:usedElsewhere}

We concentrate on the outcomes of Section~\ref{sec:models} that will be applied in Section~\ref{sec:mainResults}. The first outcome is Theorem~\ref{thm:main1}. We formulate a slightly more general result from which Theorem~\ref{thm:main1} immediately follows (we recall that the notion of a simple projectional skeleton was defined in Section~\ref{sec:prelim}; however, the fact that the constructed skeleton is even simple is just an additional feature related to the notion from \cite{CK14} but not used in the paper any further).

\begin{thm}\label{thm:subskeletons}
Let $X$ be a Banach space and $\{\mathfrak{s}_n\colon n \in \omega\}$ be a countable family of projectional skeletons on $X$ inducing the same set $D\subset X^*$. Then there exists a simple projectional skeleton on $X$ which is isomorphic to a subskeleton of $\mathfrak{s}_n$, for every $n\in\omega$.

Moreover, for every countable set $S$ and every finite list of formulas $\Phi$, there exists a family $\M$ consisting of countable suitable models for $\Phi$ and $S$ such that every $M\in\M$ admits canonical projection $P_M$ associated to $X$ and $D$, and
if $\M$ is ordered by inclusion, then the following holds:
\begin{itemize}
    \item $\F = \{X_M\colon M\in \M\}$ is a rich family and if we set $P_F:=P_M$, whenever $F = X_M$ with $M\in\M$, then  $(P_F)_{F\in\F}$ is a simple projectional skeleton on $X$ isomorphic to $(P_M)_{M\in\M}$, via the order isomorphism $\M\ni M\longmapsto X_M\in \F$.
    \item $(P_M)_{M\in\M}$ is isomorphic to a subskeleton of $\mathfrak{s}_n$, for every $n\in\omega$.
\end{itemize}
\end{thm}
\begin{proof} Let $\kappa=\dens{X}$ and $d\colon \kappa \to X$ be such that $\{d(i)\colon i<\kappa\}$ is dense in $X$. Fix a countable set $S$ and a finite list of formulas $\Phi$. Let $\Phi'$ be a finite list of formulas closed on taking subformulas that contains $\Phi$ and the union of the lists of formulas from the statements of Proposition~\ref{p:characterizationPSkeletonNotCountableModels}, Proposition~\ref{prop:inclusions} and \cite[Lemma~7]{CCS}.
Let $S'$ be a countable set such that $S \subset S'$, $S'$ contains $X$, $d$, $\mathfrak{s}_n$, for every $n \in \omega$ and $S'$ contains the union of the countable sets from the statements of Proposition~\ref{p:characterizationPSkeletonNotCountableModels}, Proposition~\ref{prop:inclusions} and \cite[Lemma~7]{CCS}.
By \cite[Theorem~IV.7.4]{kunenBook}, there exists a set $R$ such that $R\prec(\Phi';S'\cup \kappa)$. Let $\psi\colon \PP(R)\to\PP(R)$ be the Skolem function given by Lemma~\ref{l:skolem}. Consider the following family of countable subsets of $R$
\[
    \M:=\{\psi(A\cup S')\colon A\in[\kappa]^{\leq\omega}\}.
\]
Note that for every $M\in\M$ it holds that $M\prec (\Phi';S')$ and therefore it follows from Proposition~\ref{p:characterizationPSkeletonNotCountableModels} that $M$ admits canonical projection $P_M$ associated to $X$ and $D$. In order to see that the mapping $\M\ni M\mapsto X_M\in \F$ is an order-isomorphism, if $\M$ and $\F$ are ordered by inclusion, note that it follows from Lemma~\ref{l:skolem}\eqref{it:construction} applied to $J=\kappa\cup S'$ that for every $M,N\in\M$, it holds that $M\subset N$ if and only if $M\cap \kappa\subset N\cap \kappa$, which is by Proposition~\ref{prop:inclusions} equivalent to $X_M\subset X_N$.

Now let us show that $\F$ is a rich family. If $Y\subset X$ is a separable subspace, there is $A\in[\kappa]^{\leq \omega}$ such that $Y\subset \overline{\{d(i)\colon i\in A\}}$. By \cite[Lemma~7(2)]{CCS} we have that $\overline{\{d(i)\colon i\in A\}}\subset X_{\psi(A\cup S')}$. Hence $\F$ is cofinal. Now let $(M_n)_{n\in\en}$ be an increasing sequence in $\M$ and set $M_\infty:=\bigcup_{n \in \en} M_n$. It follows from Lemma~\ref{l:skolem}\eqref{it:union} that $M_\infty$ belongs to $\M$ and it is easy to check that $X_{M_\infty} = \overline{\bigcup_{n\in\en} X_{M_n}}$. Thus, $\F$ is a rich family.

Moreover, if $M,N \in \M$ satisfy $X_M\subset X_N$, then $P_M(X)=X_M\subset X_N=P_N(X)$ and $\ker P_M = (D\cap M)_\perp\supset (D\cap N)_\perp = \ker P_N$, which implies that $P_M = P_N\circ P_M = P_M\circ P_N$. Therefore, we conclude that $(P_F)_{F\in \F}$ is a simple projectional skeleton on $X$ that is isomorphic to $(P_M)_{M \in \M}$.

Finally, fix any $n\in\en$ and let us show that $(P_M)_{M \in \M}$ is isomorphic to a subskeleton of $\mathfrak{s}_n$. Consider the mapping $\phi\colon \M\to\Gamma_n$ given by $\phi(M)=\sup(\Gamma_n\cap M)$, where $\mathfrak{s}_n=(P^n_s)_{s \in \Gamma_n}$. Note that $P_M = P^n_{\phi(M)}$, for every $M \in \M$. Indeed, since $\Gamma_n\cap M$ is countable and up-directed, there is an increasing sequence $(s_k)_{k\in\omega}$ from $\Gamma_n\cap M$ with $\sup_k s_k = s = \phi(M)$. Then, using that the sequence $\{s_k\colon k\in\omega\}$ is cofinal in $\Gamma_n\cap M$, we obtain that $P_M = P_{\{s_k\colon k\in\omega\}} = P_{\phi(M)}^n$ as claimed above. To see that $\phi$ is an order-isomorphism onto its image, fix $M,N \in \M$. If $M\subset N$, then it is clear that $\phi(M)\leq \phi(N)$. On the other hand, if $\phi(M)\leq \phi(N)$, then $X_M = P_M[X] = P_{\phi(M)}^n[X]\subset P_{\phi(N)}^n[X] = X_N$ and thus $M\subset N$. It remains to show that $\phi[\M]$ is a $\sigma$-closed subset of $\Gamma_n$. Let $(M_k)_{k \in \en}$ be an increasing sequence in $\M$. As discussed previously we have that $M_\infty:=\bigcup_{n \in \en} M_n \in \M$ and it is easy to see that $\phi(M_\infty)=\sup_{n \in \en} \phi(M_n)$.
\end{proof}

\begin{remark}\label{ourresult-generalizes}
As mentioned above, Theorem ~\ref{thm:subskeletons} is a generalization of \cite[Theorem 4.1]{CK14}. Note however that in the proof of \cite[Theorem 4.1]{CK14} there is a gap which was fixed only later. Namely, the authors used the fact that any Banach space with a projectional skeleton has a Markushevich basis, which was not known at the time and it was proved later in \cite{kal20}. Our proof of Theorem~\ref{thm:subskeletons} is different from the proof of \cite[Theorem 4.1]{CK14} - the difference is that, inspired by the proof of \cite[Theorem 4]{C18}, instead of the existence of Markushevich basis we use the deep result by H. Toru\'{n}czyk that all infinite-dimensional Banach spaces of the same density are topologically homeomorphic.
\end{remark}

The following seems to be the crucial property of projectional skeletons which we need further in our inductive arguments.

\begin{cor}\label{cor:plichko}
Let $X$ be a Banach space, $\mathfrak{s}$ be a projectional skeleton on $X$ and $A$ be a subset of $X$ that countably supports $D(\mathfrak{s})$. Then there exists a projectional skeleton $(P_s)_{s\in \Gamma}$ on $X$ that is isomorphic to a subskeleton of $\mathfrak{s}$ and such that $P_s(x) \in \{0,x\}$, for every $x \in A$ and every $s \in \Gamma$.
\end{cor}
\begin{proof}
This is a consequence of Lemma~\ref{lem:countablysupports} and Theorem~\ref{thm:subskeletons}.
\end{proof}

\begin{remark}
Note that if there exist a projectional skeleton $(P_s)_{s\in\Gamma}$ on a Banach space $X$ and a linearly dense set $A\subset X$ such that $\{P_sx\colon s\in\Gamma\}\subset\{0,x\}$ for every $x\in A$, then we easily obtain that $P_s P_tx=P_t P_sx$, for every $x\in A$, which implies that the skeleton is commutative. Therefore, it seems that our methods apply to subclasses of Plichko spaces (equivalently, as mentioned in Section~\ref{sec:prelim}, spaces admitting a commutative projectional skeleton).
\end{remark}

In the following technical result, we construct a transfinite sequence of projections using the method of suitable model. An alternative way of constructing this sequence would be to use the approach from \cite[Section 2]{kal20}. One could deduce many other properties of the projections constructed in the proof of Proposition~\ref{prop:pri}, but we formulate only the ones which will be used further (we recall that shrinkingness-related sets $\T(\mathfrak{s}, A)$ were defined in Section~\ref{sec:prelim}).

\begin{prop}\label{prop:pri}
Let $X$ be a Banach space and put $\kappa:=\dens X$. Let $A\subset X$ be a bounded set and let $\mathfrak{s}=(P_s)_{s\in\Gamma}$ be a projectional skeleton on $X$ such that $A$ countably supports $D(\mathfrak{s})$ and $P_s x\in\{0,x\}$, for every $x\in A$ and
every $s \in \Gamma$. Then there exists a long sequence of bounded projections $\mathfrak{r} = (R_\alpha)_{\alpha\leq\kappa}$ satisfying the following properties

\begin{enumerate}[label=(R\arabic*)] 
\newcounter{saveenum} 
    \item\label{it:P0-Pkappa} $R_0=0$ and $R_\kappa=\Id$.
    \item\label{it:comutative} For every $\alpha<\beta$, we have that $R_\alpha\circ R_\beta = R_\beta\circ R_\alpha = R_\alpha$.
     \item\label{it:upDirectedLimit} Let $\alpha\le\kappa$, let $\eta\colon[0,\alpha)\to \kappa$ be an increasing function and let $\xi \le\kappa$ be a limit ordinal with $\sup_{\beta<\alpha} \eta(\beta) = \xi$. Then $\lim_{\beta < \alpha} R_{\eta(\beta)}(x) = R_\xi(x)$, for every $x\in X$.
     \item\label{it:skeletonOnSubset} For every $\alpha\in[0,\kappa)$, we have that $\dens(R_\alpha[X]) < \kappa$.
     \item\label{it:niceCanonicalPRI} $R_{\alpha}x\in\{0,x\}$, for every $\alpha\le\kappa$ and $x\in A$;
      \item \label{it:subskeletonOnDifferences} For every $\alpha < \kappa$ there exists a $\sigma$-closed subset $\Gamma_\alpha\subset \Gamma$ such that the family $\mathfrak{s}_\alpha = (P_s|_{(R_{\alpha+1}-R_\alpha)[X]})_{s\in \Gamma_\alpha}$ is a projectional skeleton on $(R_{\alpha+1}-R_\alpha)[X]$.

      Moreover, $(R_{\alpha+1}-R_\alpha)[A]$ countably supports $D(\mathfrak{s}_\alpha)$.
     \item\label{it:AeShrink} $\T(\mathfrak{s}, A)\subset \T(\mathfrak{r}, A)$.
    \setcounter{saveenum}{\value{enumi}} 
\end{enumerate}
\end{prop}
\begin{proof}
Let $S$ be union of the countable sets from the statements of Proposition~\ref{p:characterizationPSkeletonNotCountableModels} and Lemma \ref{lem:countablysupports} and let $\Phi$ be the union of the finite lists of formulas from the statements of Proposition~\ref{p:characterizationPSkeletonNotCountableModels} and Lemma~\ref{lem:countablysupports} enriched by the finitely many formulas used in the proof below. Let $\{d_\alpha\colon \alpha<\kappa\}$ be a dense set in $X$. Using Theorem~\ref{thm:modelExists}, we construct a transfinite sequence of suitable models $(M_\alpha)_{\alpha\leq \kappa}$ such that for every $\alpha\in[0,\kappa]$ we have
\begin{itemize}
    \item $\vert M_\alpha \vert \le \max(\omega, \vert \alpha \vert)$, $M_\alpha\prec (\Phi;S \cup \{A\})$ and $M_\alpha$ contains $X$ and $\mathfrak{s}$,
    \item $M_{\alpha+1}\supset M_\alpha\cup\{d_\alpha\}$ whenever $\alpha\neq\kappa$,
    \item $M_\alpha = \bigcup_{\beta<\alpha} M_\beta$, if $\alpha$ is a limit ordinal (in order to see that an increasing union of suitable models is a suitable model we use e.g. \cite[Lemma 2.1]{C12}).
\end{itemize}
Set $R_0:=0$ and for $\alpha\in(0,\kappa]$ we let $R_\alpha$ be the canonical projection associated to $M_\alpha$, $X$ and $D(\mathfrak{s})$ given by Proposition~\ref{p:characterizationPSkeletonNotCountableModels}, that is, it holds that $R_{\alpha}[X] = X_{M_\alpha}$, $\ker R_{\alpha} = (D(\mathfrak{s})\cap M_\alpha)_\perp$ and
\begin{equation}\label{eq:formulaUsingLimit}
R_{\alpha}x = \lim_{s\in \Gamma\cap M_\alpha} P_s x,\qquad x\in X,
\end{equation}
where $\Gamma\cap M_\alpha$ is an up-directed set. Since $M_\kappa\supset \{d_\alpha\colon \alpha<\kappa\}$, we have that $R_\kappa[X] = X$ and thus \ref{it:P0-Pkappa} holds. Note that \ref{it:comutative} follows immediately from \cite[Lemma 2.2]{kal20}, since the sequence $(\Gamma\cap M_\alpha)_{\alpha\leq \kappa}$ is increasing. In order to prove \ref{it:upDirectedLimit}, pick $\alpha$, $\eta$ as in \ref{it:upDirectedLimit} and fix $x\in X$ and $\varepsilon>0$. Find $s_0\in \Gamma\cap M_\xi$ such that for every $s\geq s_0$, $s\in M_\xi\cap \Gamma$ we have $\|P_sx - R_\xi x\|\leq \varepsilon$. Since $M_\xi = \bigcup_{\beta<\alpha} M_{\eta(\beta)}$, there is $\beta_0<\alpha$ with $s_0\in M_{\eta(\beta_0)}$. Then for every $\beta \in [\beta_0,\alpha)$ we have
\[
\|R_{\eta(\beta)}(x) - R_{\xi}x\| = \lim_{s\in \Gamma\cap M_{\eta(\beta)},s\geq s_0} \|P_s(x) - R_{\xi}x\| \leq \varepsilon,
\]
which implies that $\lim_{\beta<\alpha} R_{\eta(\beta)}(x) = P_{M_\xi}(x)=R_\xi(x)$ and so \ref{it:upDirectedLimit} holds. Clearly \ref{it:skeletonOnSubset} follows from the fact that $R_\alpha[X] = X_{M_\alpha}$ and $|M_\alpha| < \kappa$ for $\alpha<\kappa$ and \ref{it:niceCanonicalPRI} is a consequence of the formula \eqref{eq:formulaUsingLimit}.

In order to prove \ref{it:subskeletonOnDifferences}, fix $\alpha < \kappa$. By \cite[Proposition 2.3 (iii)]{kal20}, the family $(P_s|_{R_{\alpha+1}[X]})_{s\in (\Gamma\cap M_{\alpha+1})_\sigma}$ is a projectional skeleton on $R_{\alpha+1}[X]$ with induced set given by \[D_{\alpha + 1} := \bigcup_{s\in (\Gamma\cap M_{\alpha+1})_\sigma}\{P_s^*x^*|_{R_{\alpha+1}[X]}\colon x^*\in X^*\}.\]
Observe that it follows from absoluteness that given $d\in D(\mathfrak{s})\cap M_\alpha$ there is $s\in \Gamma\cap M_{\alpha+1}$ such that $d\in P_s^*X^*$ and thus $d|_{R_{\alpha+1}[X]}\in D_{\alpha+1}$. This implies that \[(R_{\alpha+1}-R_\alpha)[X] = R_{\alpha+1}[X]\cap \ker R_\alpha = \bigcap_{d\in D(\mathfrak{s})\cap M_\alpha} \ker d|_{R_{\alpha+1}[X]}\]
is $w(R_{\alpha+1}[X], D_{\alpha+1})$-closed and then the existence of $\Gamma_\alpha$ as in \ref{it:subskeletonOnDifferences} follows from \cite[Lemma 2.1]{kal20}. The moreover part of \ref{it:subskeletonOnDifferences} follows from \ref{it:niceCanonicalPRI}, since $(R_{\alpha+1}-R_\alpha)[A]\subset A\cup\{0\}$.

Finally let us prove \ref{it:AeShrink}. Let $\varepsilon>0$, $B\subset A$ and $x^*\in X^*$ be such that $\mathfrak{s}$ is $(B,\varepsilon)$-shrinking for $x^*$, pick an increasing sequence $(\alpha_k)_{k\in\omega}$ in $[0,\kappa)$ and put $\alpha := \sup_k \alpha_k$. In order to get a contradiction assume that $\limsup_k \rho_{B}(R_{\alpha_k}^*(x^*),R_{\alpha}^*(x^*)) > \varepsilon\|x^*\|$ in which case, up to passing to a subsequence, we may assume that there is $\delta>0$ such that
 \[
 \rho_B(R_{\alpha_k}^*(x^*),R_{\alpha}^*(x^*))> \varepsilon \|x^*\| + \delta, \quad \text{for every }k\in\omega.
 \]
Fix any $k\in\omega$ for a while. 
 We have that
 \[
    \rho_B(R_{\alpha_k}^*(x^*),R_{\alpha}^*(x^*))=\sup_{x\in B}|x^*(R_{\alpha_k}(x))-x^*(R_{\alpha}(x))|,
 \]
 which implies that there exists $x_k\in B$ such that \[|x^*(R_{\alpha_k}(x_k))-x^*(R_{\alpha}(x_k))|\geq\varepsilon \|x^*\| +\delta.\] Putting this together with \ref{it:niceCanonicalPRI} we obtain that $|x^*(x_k)|\geq\varepsilon \|x^*\| +\delta$.
Note that it follows from \ref{it:comutative} that $R_{\alpha}(x_k)=x_k$ and $R_{\alpha_k}(x_k)=0$. Now, using that $P_sx\in\{0,x\}$, for every $s\in\Gamma$ and $x\in B$, we may recursively construct two increasing sequences $(s_j)_{j\in\omega}$ in $M_{\alpha}\cap\Gamma$ and $(k_j)_{j\in\omega}$ in $\mathbb{N}$ 
 having the following properties
 \begin{itemize}
     \item $s_{2j+1}\in M_{\alpha_{k_j}}$ and $P_{s_{2j+1}}x_{k_j}=0$, for every $j\in\omega$;
     \item $s_{2j+2}\in M_{\alpha_{k_{j+1}}}$ and $P_{s_i}x_{k_j}=x_{k_j}$, for every $j\in\omega$ and $i\geq 2j+2$.
 \end{itemize}
 Indeed, in the initial step of the induction we put $k_0 = 1$. If $k_0,\ldots,k_j$ are defined, using \eqref{eq:formulaUsingLimit} we find $s_{2j+1}\in M_{\alpha_{k_j}}$ such that $P_{s_{2j+1}}x_{k_j} = R_{\alpha_{k_j}} x_{k_j} = 0$ and we find $s_{2j+2}\in M_{\alpha}\cap\Gamma(=\bigcup_{k\in\omega}M_{\alpha_k}\cap\Gamma)$ such that $s_{2j+2}\geq s_{2j+1}$ and $P_{s_{2j+2}}x_{k_j}=R_\alpha x_{k_j} = x_{k_j}$ (and therefore $P_{s'}x_{k_j} = x_{k_j}$ for any $s'\geq s_{2j+2}$). Finally, we pick $k_{j+1}>k_j$ such that $s_{2j+2}\in M_{\alpha_{k_{j+1}}}$, which finishes the inductive step.

Putting $s:=\sup_j s_j$, we have that $P_{s_{2j+1}}x_{k_j}=0$ and  $P_sx_{k_j}=x_{k_j}$ for every $j\in\omega$. Therefore, for every $j\in\omega$, we have that
 \[
 \rho_B(P_{s_{2j+1}}^*(x^*),P_s^*(x^*))\geq|x^*(x_{k_j})|\geq \varepsilon\|x^*\| + \delta.
 \]
But this contradicts the fact that $\mathfrak{s}$ is $(B,\varepsilon)$-shrinking in $x^*$. Thus, \ref{it:AeShrink} holds.
\end{proof}

\section{Main results}\label{sec:mainResults}

This section contains the proofs of our main results. The main technical parts are concentrated in Subsection~\ref{subsec:shrinking} and the proofs of our characterizations are then presented in the remainder of the section.

\subsection{Shrinkingness passes from projectional skeletons to SPRI}\label{subsec:shrinking}

Frist we aim here at proving Theorem~\ref{thm:constructionSPRI}, which enables us to pass shrinking-like properties from projectional skeletons to SPRI.

\begin{lemma}\label{lem:spricase}
Let $X$ be a Banach space with a projectional skeleton $\mathfrak{s}$ and put $\kappa:=\dens(X)$. Let $A\subset X$ be a bounded set and let $\mathfrak{p} = (R_{\alpha})_{\alpha\leq \kappa}$ be a long sequence of bounded projections satisfying \ref{it:P0-Pkappa}, \ref{it:comutative}, \ref{it:upDirectedLimit},  \ref{it:niceCanonicalPRI} and \ref{it:AeShrink} from Proposition~\ref{prop:pri}.
Assume that for every $\alpha<\kappa$ the space $(R_{\alpha+1}-R_{\alpha})[X]$ is either separable or it admits a SPRI $\mathfrak{S}_\alpha=(Q_{\beta}^{\alpha})_{\beta\leq \mu_{\alpha}}$,  which satisfies:
\begin{enumerate}[resume]
   \item\label{it:partialNiceSPRI} $Q^{\alpha}_{\beta}x\in\{0,x\}$, for every  $x\in (R_{\alpha+1}-R_{\alpha})[A]$ and every $\beta\in[0,\mu_\alpha]$;
   \item\label{it:SPRIshrinkPartial} for every $(\varepsilon,B,x^*)\in\T(\mathfrak{s},A)$, we have that
   \[(\varepsilon_\alpha, B_\alpha, x^*_\alpha)\in\T(\mathfrak{S}_\alpha, A_\alpha),\]
   where $A_\alpha:=(R_{\alpha+1}-R_{\alpha})[A]$, $B_\alpha:=(R_{\alpha+1}-R_{\alpha})[B]$, $x^*_\alpha:=x^*|_{ (R_{\alpha+1}-R_{\alpha})[X]}$, $\varepsilon_\alpha: = \varepsilon\tfrac{\|x^*\|}{\|x^*_\alpha\|}$, if $x^*_\alpha \ne 0$ and $\varepsilon_\alpha:=0$, if $x^*_\alpha=0$.
\end{enumerate}
Then the whole space $X$ admits a SPRI $\mathfrak{S} = (Q_{\alpha})_{\alpha\leq \kappa}$ such that
\begin{enumerate}[resume]
   \item\label{it:SPRInicetotal} $Q_{\alpha}x\in\{0,x\}$, for every  $x\in A$ and $\alpha\in[0,\kappa]$;
   \item\label{it:SPRIshrink} $\T(\mathfrak{s}, A)\subset \T(\mathfrak{S}, A)$.
\end{enumerate}
\end{lemma}

\begin{proof} We argue as in \cite[Proposition 6.2.7]{FabRedBook}. Let $\alpha<\kappa$. If $(R_{\alpha+1}-R_{\alpha})[X]$ is separable, then we put $\mu_{\alpha}=\omega$, $Q_{0}^{\alpha}\equiv 0$, and $Q_{\beta}^{\alpha}=R_{\alpha+1}-R_{\alpha}$, for all $0<\beta<\omega$. If $(R_{\alpha+1}-R_{\alpha})[X]$ is nonseparable, let $(Q_{\beta}^{\alpha})_{\beta\leq\mu_{\alpha}}$ be the SPRI given by the hypothesis. We start by defining $Q_{(\kappa,0)}:=\Id$,
 \[
 Q_{(\alpha,\beta)}:=Q_{\beta}^{\alpha}(R_{\alpha+1}-R_{\alpha})+ R_{\alpha}, \,\, 0\leq\beta<\mu_{\alpha}, \, 0\leq \alpha<\kappa,
 \]
 \[
 \Lambda:=\{(\alpha,\beta)\colon \beta<\mu_{\alpha}, \alpha<\kappa\}\cup\{(\kappa,0)\}.
 \]
We endow $\Lambda$ with the lexicographical order, i.e $(\alpha,\beta)\leq(\alpha',\beta')$ if and only if $\alpha<\alpha'$ or $\alpha=\alpha'$ and $\beta\leq\beta'$. Note that $\Lambda$ is order-isomorphic to $[0,\kappa]$ and that the elements of $\Lambda$ that correspond to limit ordinals are of following forms: $(\alpha,0)$, for any $\alpha>0$ and $(\alpha,\beta)$, for arbitrary $\alpha$ and limit $\beta$. With the same proof as the one presented in \cite[Proposition 6.2.7]{FabRedBook}, we conclude that $\mathfrak{S}:=(Q_{(\alpha,\beta)})_{(\alpha,\beta)\in \Lambda}$ is a SPRI on $X$ (up to reindexation using $\Lambda\cong[0,\kappa]$). Let us show that $\mathfrak{S}$ satisfies the additional properties \eqref{it:SPRInicetotal} and \eqref{it:SPRIshrink}.

 \medskip
 \noindent \eqref{it:SPRInicetotal}: Fix $x\in A$ and $\alpha \in$ $[0,\kappa)$. Note that by \ref{it:comutative} and \ref{it:niceCanonicalPRI} we have that $(R_{\alpha+1}-R_{\alpha})x\in \{0,x\}$. Thus, if $\mu_\alpha = \omega$, then \eqref{it:SPRInicetotal} follows. Otherwise, if $R_{\alpha}x=0$, then since $(R_{\alpha+1}-R_{\alpha})x\in (R_{\alpha+1}-R_{\alpha})[A]$, applying \eqref{it:partialNiceSPRI} we obtain that  $Q_{\beta}^{\alpha}(R_{\alpha+1}-R_{\alpha})(x)\in\{0,(R_{\alpha+1}-R_{\alpha})(x)\} = \{0,x\}$, therefore $Q_{(\alpha,\beta)}x\in\{0,x\}$. On the other hand if $R_{\alpha}x=x$, then $R_{\alpha+1}x=x$ as well and so $Q_{(\alpha,\beta)}x=R_{\alpha}x=x$. This proves \eqref{it:SPRInicetotal}.

 \medskip

\noindent\eqref{it:SPRIshrink}: Fix $(\varepsilon,B,x^*)\in\T(\mathfrak{s},A)$ and let $((\alpha_{k},\beta_{k}))_{k\in\omega}$ be an increasing sequence in $\Lambda$ such that $\sup_{k \in \omega}(\alpha_k,\beta_k)=(\alpha,\beta)$.  Only three cases are possible.

\medskip
\noindent\emph{Case 1: $\beta \neq 0$}.

\medskip
\noindent
 In this case $\alpha_k$ is eventually equal to $\alpha$ and $\sup_{k} \beta_k=\beta<\mu_{\alpha}$. If $\mu_\alpha=\omega$, then the result follows easily. Now assume that $\mu_\alpha>\omega$ and note that for a fixed $x \in B$, we have that

     \begin{equation*}
     \begin{split}       |Q_{(\alpha,\beta_k)}^*(x^*)(x) & - Q_{(\alpha,\beta)}^*(x^*)(x)|\\ &=|(Q^\alpha_{\beta_k})^*(x^*)((R_{\alpha+1}-R_{\alpha})(x))-(Q^\alpha_{\beta})^*(x^*)((R_{\alpha+1}-R_{\alpha})(x))|\\
         &\leq \rho_{B_\alpha}((Q^\alpha_{\beta_k})^*(x^*_\alpha),(Q^\alpha_{\beta})^*(x^*_\alpha)).
     \end{split}
     \end{equation*}
     Thus,
     \[
     \rho_{B}(Q_{(\alpha,\beta_k)}^*(x^*),Q_{(\alpha,\beta)}^*(x^*))
     \leq \rho_{B_\alpha}((Q^\alpha_{\beta_k})^*(x^*_\alpha),(Q^\alpha_{\beta})^*(x^*_\alpha))
     \]
     and by \eqref{it:SPRIshrinkPartial} we get
     \[
     \limsup_k\rho_{B}(Q_{(\alpha,\beta_k)}^*(x^*),Q_{(\alpha,\beta)}^*(x^*))\leq\varepsilon_\alpha \|x^*_\alpha\| \le \varepsilon\|x^*\|.
     \]
     Therefore we conclude that $(\varepsilon,B,x^*)\in \T(\mathfrak{S},A)$.

\medskip
\noindent\emph{Case 2: $\beta = 0$ and $\alpha$ is limit}.

\medskip
\noindent
In this case $\alpha = \sup_k \alpha_k$. Fix $x\in B$ and recall that $Q_{(\alpha,0)} = R_\alpha$. It follows from \ref{it:upDirectedLimit} and \ref{it:niceCanonicalPRI} that $(R_{\alpha_k+1}-R_{\alpha_k})(x)$ is eventually equal to zero for every $x\in B$. Therefore we get for every $x\in B$
     \begin{equation*}
         \begin{split}
         |Q_{(\alpha_k,\beta_k)}^*(x^*)(x) & - Q_{(\alpha,0)}^*(x^*)(x)|\\&=|(Q^{\alpha_k}_{\beta_k})^*(x^*)((R_{\alpha_k+1}-R_{\alpha_k})(x))+x^*(R_{\alpha_k}x)-x^*(R_{\alpha}x)|\\
         &= |x^*(R_{\alpha_k}x)-x^*(R_{\alpha}x)|\leq \rho_{B}(R_{\alpha_k}^*(x^*),R_{\alpha}^*(x^*)),
         \end{split}
     \end{equation*}
     for all $k \in \omega$ big enough. From this we get that
     \[
     \rho_{B}(Q_{(\alpha_k,\beta_k)}^*(x^*),Q_{(\alpha,\beta)}^*(x^*))
     \leq \rho_{B}(R_{\alpha_k}^*(x^*),R_\alpha^*(x^*)).
     \]
     Note that \ref{it:AeShrink} ensures that $(\varepsilon,B,x^*)\in\T(\mathfrak{p},A)$ and thus we obtain that
     \[
     \limsup_k\rho_{B}(Q_{(\alpha_k,\beta_k)}^*(x^*),Q_{(\alpha,\beta)}^*(x^*)) \le\varepsilon \|x^*\|.
     \]
     This shows that $(\varepsilon,B,x^*)\in \T(\mathfrak{S},A)$.

\medskip
\noindent\emph{Case 3: $\beta = 0$ and $\alpha = \gamma+1$, for some $\gamma<\kappa$}.

\medskip
\noindent
In this case $\alpha_k$ is eventually equal to $\gamma$ and $\sup_k\beta_k=\mu_{\gamma}$. So we may assume that $\alpha_k=\gamma$, for every $k\in\omega$. Fix $x\in B$ and recall that $Q_{(\gamma+1,0)} = R_{\gamma+1}$. If $\mu_\gamma=\omega$, then we observe that $Q_{(\gamma,\beta_k)} = R_{\gamma+1}$ whenever $\beta_k\neq 0$, so eventually we have $Q_{(\alpha_k,\beta_k)} = Q_{(\alpha,\beta)}$. Now assume that $\mu_\gamma>\omega$. Using that $Q_{\mu_{\gamma}}^{\gamma}$ is the identity map on $(R_{\gamma+1}-R_{\gamma})[X]$, we obtain
      \begin{equation*}
      \begin{split}
         |Q_{(\gamma,\beta_k)}^* & (x^*)(x) - Q_{(\gamma+1,0)}^*(x^*)(x)|\\ &=|(Q^\gamma_{\beta_k})^*(x^*)((R_{\gamma+1}-R_{\gamma})(x))-x^*((R_{\gamma+1}-R_{\gamma})(x)))|\\
         &=|(Q^\gamma_{\beta_k})^*(x^*)((R_{\gamma+1}-R_{\gamma})(x))-(Q^\gamma_{\mu_{\gamma}})^*(x^*)((R_{\gamma+1}-R_{\gamma})(x))|\\
         &\leq\rho_{B_\gamma}((Q^\gamma_{\beta_k})^*(x^*_\gamma),(Q^\gamma_{\mu_{\gamma}})^*(x^*_\gamma)).
     \end{split}
     \end{equation*}
     Therefore
     \[
     \rho_{B}(Q_{(\alpha_k,\beta_k)}^*(x^*),Q_{(\alpha,0)}^*(x^*))
     \leq \rho_{B_\gamma}((Q^\gamma_{\beta_k})^*(x^*_\gamma),(Q^\gamma_{\mu_{\gamma}})^*(x^*_\gamma))
     \]
     and using \eqref{it:SPRIshrinkPartial}, we obtain that
     $    \limsup_{k \in \omega}\rho_{B}(Q_{(\alpha_k,\beta_k)}^*(x^*),Q_{(\alpha,0)}^*(x^*))\leq \varepsilon \|x^*\|$.
     This shows that $(\varepsilon,B,x^*)\in \T(\mathfrak{S},A)$.
\end{proof}

\begin{thm}\label{thm:constructionSPRI}
Let $X$ be a nonseparable Banach space with a  projectional skeleton $\mathfrak{s} = (P_s)_{s\in\Gamma}$ and put $\kappa:=\dens(X)$. Let $A$ be a bounded subset of $X$ countably supporting $D(\mathfrak{s})$. Then $X$ admits a SPRI $\mathfrak{S} = (Q_\alpha)_{\alpha\leq\kappa}$ satisfying the following conditions
\begin{enumerate}[label=(\alph*)]
    \item\label{it:niceSPRI} For every  $x\in A$ and $\alpha\in[0,\kappa]$, $Q_{\alpha}x\in\{0,x\}$;
    \item\label{it:SPRIsubshrink} $\T(\mathfrak{s},A)\subset \T(\mathfrak{S},A)$.
\end{enumerate}
\end{thm}
\begin{proof}
  According to Corollary~\ref{cor:plichko} we may and do assume that $P_sx\in \{0,x\}$, for every $x\in A$ and $s\in\Gamma$. We will prove the theorem by induction on $\kappa\geq \omega_1$. Let $\mathfrak{p} = (R_\alpha)_{\alpha\leq\kappa}$ be the long sequence of bounded projections satisfying \ref{it:P0-Pkappa}-\ref{it:AeShrink} given by Proposition~\ref{prop:pri}. If $\kappa=\omega_1$, then using \ref{it:P0-Pkappa}, \ref{it:comutative}, \ref{it:upDirectedLimit} and \ref{it:skeletonOnSubset}, we conclude that the family $\mathfrak{S} := \mathfrak{p}$ is a SPRI on $X$. Moreover, it follows from \ref{it:niceCanonicalPRI} and \ref{it:AeShrink} that $\mathfrak{S}$ satisfies conditions \ref{it:niceSPRI} and \ref{it:SPRIsubshrink}.

  Assume now that $\kappa>\omega_1$ and that the theorem holds for every nonseparable Banach space with density strictly smaller than $\kappa$. Let us verify that $\mathfrak{p}$ satisfies the assumptions of Lemma \ref{lem:spricase}. Fix $\alpha < \kappa$ and suppose that $(R_{\alpha+1}-R_{\alpha})[X]$ is not separable. By \ref{it:subskeletonOnDifferences}, there exists a $\sigma$-closed subset $\Gamma_\alpha$ of $\Gamma$ such that $\mathfrak{s}_\alpha:=(P_s|_{(R_{\alpha+1}-R_\alpha)[X]})_{s\in \Gamma_\alpha}$ is a projectional skeleton on $(R_{\alpha+1}-R_\alpha)[X]$ and $A_\alpha:=(R_{\alpha+1}-R_{\alpha})[A]$ countably supports $D(\mathfrak{s_\alpha})$. Therefore it follows from \ref{it:skeletonOnSubset} and the induction hypothesis that $(R_{\alpha+1}-R_{\alpha})[X]$ admits a SPRI $\mathfrak{S}_\alpha$ satisfying the appropriate versions of \ref{it:niceSPRI} and \ref{it:SPRIsubshrink}. It is clear that $\mathfrak{S}_\alpha$ satisfies conditions \eqref{it:partialNiceSPRI} from Lemma \ref{lem:spricase}. In order to check that $\mathfrak{S}_\alpha$ satisfies condition \eqref{it:SPRIshrinkPartial} from Lemma~\ref{lem:spricase}, fix $(\varepsilon,B,x^*)\in \T(\mathfrak{s},A)$, let $(s_n)_{n \in \omega}$ be an increasing sequence in $\Gamma_\alpha$ and put $s = \sup_n s_n$. Using that \ref{it:niceCanonicalPRI} ensures that $B_\alpha \subset B \cup \{0\}$, we obtain that
  \[\begin{split}
  \limsup_n \; &  \rho_{B_\alpha}\big((P_{s_n}|_{(R_{\alpha+1}-R_{\alpha})[X]})^*(x^*_\alpha), (P_{s}|_{(R_{\alpha+1}-R_{\alpha})[X]})^*(x^*_\alpha)\big)
 \\ & =\limsup_n \sup_{x\in B}\Big|P_{s_n}^*(x^*((R_{\alpha+1}-R_{\alpha})x))-P_{s}^*(x^*(R_{\alpha+1}-R_{\alpha}x))\Big|\\
  & \leq \limsup_n \rho_{B}(P_{s_n}^*(x^*),P_{s}^*(x^*))\leq \varepsilon\|x^*\|,
  \end{split}\]
  which implies that $(\varepsilon_\alpha,B_\alpha,x^*_\alpha)\in \T(\mathfrak{s}_\alpha, A_\alpha)$. Since the induction hypothesis ensures that  $\T(\mathfrak{s}_\alpha, A_\alpha) \subset \T(\mathfrak{S}_\alpha, A_\alpha)$, we conclude that $(\varepsilon_\alpha,B_\alpha,x^*_\alpha)\in \T(\mathfrak{S}_\alpha, A_\alpha)$. Thus $\mathfrak{S}_\alpha$ satisfies condition \eqref{it:SPRIshrinkPartial} from Lemma \ref{lem:spricase}. Now the result follows from this Lemma.
 \end{proof}

\subsection{Weakly $\mathcal{K}$-analytic Banach spaces}
In what follows, we denote by $\omega^{<\omega}$ the collection of functions defined on a natural number and taking values in $\omega$, i.e., $\omega^{<\omega}=\bigcup_{i\in \omega} \omega^i$.

Following \cite{FGMZ}, we say that a Banach space $X$ is \emph{weakly $\mathcal{K}$-analytic} if there exists a family $\{K_t\colon t\in \omega^{<\omega}\}$ of weak$^{*}$ compact subsets of $X^{**}$ such that $X=\bigcup_{\sigma\in\omega^{\omega}}\bigcap_{i\ge 1}K_{\sigma|_i}$. The key ingredients to establish the characterization of weakly $\mathcal{K}$-analytic Banach spaces given in Theorem \ref{thm:main} are Theorem \ref{thm:constructionSPRI} and the characterization of weakly $\mathcal{K}$-analytic spaces presented in \cite[Theorem 4]{FGMZ} that we recall in Lemma \ref{l:theorem4-FGMZ} (more precisely, using the well-known fact that any weakly $\mathcal{K}$-analytic Banach space is WLD, Lemma \ref{l:theorem4-FGMZ} can be easily deduced from \cite[Theorem 4]{FGMZ}, where a slightly different formulation is used).

\begin{lemma}\label{l:theorem4-FGMZ}
For a Banach space $X$, the following conditions are equivalent
\begin{enumerate}[label = (\roman*)]
    \item $X$ is a weakly $\mathcal{K}$-analytic space;
    \item \label{it:condition-ii} there exist a set $\A \subset B_X$ that is linearly dense in $X$ and countably supports $X^*$ and a family $\{A_t\colon t \in \omega^{<\omega}\}$ of subsets of $\A$ such that $A_\emptyset=\A=\bigcup_{\sigma \in \omega^\omega} \bigcap_{i \ge 1} A_{\sigma|_i}$ and for every $\varepsilon>0$, $x^* \in X^*$, $\sigma \in \omega^\omega$, there exists $i \in \omega$ such that
    \[\vert \{x \in A_{\sigma|_i}\colon \vert x^*(x) \vert >\varepsilon\}\vert< \omega.\]
\end{enumerate}
\end{lemma}

\noindent
Here, given $n, i \in \omega$ and $t\in \omega^{i}$, we denote by $n^\smallfrown t$ the element of $\omega^{i+1}$ given by $n^\smallfrown t(0):=n$ and if $i\ge 1$, then $n^\smallfrown t(j):=t(j-1)$, for every $1\le j \le i$.

\begin{proof}[Proof of Theorem~\ref{thm:main}]
\ref{it:Xwanalytic}$\Rightarrow$ \ref{it:PSshrinkwanalytic} Let the set $\mathcal{A}$ and the family $\{A_t\subset \A: t\in\omega^{<\omega}\}$ be those found in (ii) of Lemma \ref{l:theorem4-FGMZ}.
Since $\A$ is linearly dense in $X$ and it countably supports $X^*$, we have that $X$ is Plichko. Hence $X$ admits a projectional skeleton $\mathfrak{s}=(P_s)_{s \in \Gamma}$ and thus it follows from Corollary \ref{cor:plichko} that we may assume without loss of generality that $P_s(x) \in \{0, x\}$, for every $x \in \A$ and $s \in \Gamma$.\\
Now fix any $\varepsilon>0$, $x^*\in X^* \setminus \{0\}$ and $\sigma\in\omega^{\omega}$. Let $i \in \omega$ be such that the set $Y:=\{x\in A_{\sigma|_i}\colon |x^*(x)|> \varepsilon\|x^*\|\}$ is finite. We have to show that $\mathfrak{s}$ is $(A_{\sigma|_i},\varepsilon)$-shrinking in $x^*$. Fix an increasing sequence $(s_k)_{k\in\omega}$ in $\Gamma$, put $s\coloneqq\sup_{k\in\omega}s_k$ and note that
\begin{equation*}
    |P_{s_k}^*(x^*)(x)-P_{s}^*(x^*)(x)|=|x^*(P_{s_k}x)-x^*(P_s x)|=\begin{cases}
    0\,\, &{\rm if} \,\,P_{s_k}x=P_sx, \\
    |x^*(x)|\,\, &{\rm if} \,\,P_{s_k}x\neq P_sx,
    \end{cases}
\end{equation*}
for every $k\in\omega$ and $x\in \A$. Thus we have that
\begin{equation}\label{eq:star}
\forall x \in A_{\sigma|_i} \setminus Y\colon\qquad |P_{s_k}^*(x^*)(x)-P_{s}^*(x^*)(x)| \le \varepsilon \Vert x^* \Vert.
\end{equation}
Moreover, since $(P_{s_k}^*(x^*))_{k\in\omega}$ converges to $P_s^*(x^*)$ in the weak$^*$-topology of $X^*$ and $Y$ is finite, there exists $k_0\in \omega$ such that
\begin{equation}\label{eq:starStar}
\forall k \ge k_0, \ \forall x \in Y:\quad |P_{s_k}^*(x^*)(x)-P_{s}^*(x^*)(x)| \le\varepsilon \Vert x^* \Vert.
\end{equation}
\noindent
Now, using \eqref{eq:star} and \eqref{eq:starStar}, we conclude that $\limsup_k \rho_{A_{\sigma|_i}}(P_{s_k}^*x^*,P_s^*x^*)\leq \varepsilon\|x^*\|$. \\
\ref{it:PSshrinkwanalytic} $\Rightarrow$ \ref{it:SPRIwanalytic} If $X$ is separable, then our implicaton is trivial and in the case when $X$ is nonseparable, our implication follows directly from Theorem \ref{thm:constructionSPRI}. \\
\ref{it:SPRIwanalytic} $\Rightarrow$ \ref{it:Xwanalytic} Let $(Q_{\alpha})_{\alpha\leq\kappa}$, $\A$ and $\{A_t\subset B_X\colon t\in\omega^{<\omega}\}$ be given by (iii). Note that we may assume that $Q_{\alpha}\neq Q_{\alpha+1}$, for every $\alpha<\kappa$.
For every $\alpha < \kappa$, set $T_\alpha:=Q_{\alpha+1}-Q_\alpha$ and note that since $T_\alpha[X]$ is separable and $T_\alpha[\A]$ is linearly dense in $T_\alpha[X]$, there exists a countable set $\{v_n^\alpha\colon n \in \omega\} \subset T_\alpha[\A] \setminus \{0\}$ that is linearly dense in $T_\alpha[X]$.

Consider the family $\{B_\sigma\colon \sigma \in \omega^{<\omega}\}$ defined as follows:
\begin{itemize}
    \item $B_{n^\smallfrown t}:=A_{t}\cap\{v_n^{\alpha}\colon\alpha < \kappa\} \subset B_X$, for every $n \in \omega$ and $t \in \omega^{<\omega}$;
    \item $B_{\emptyset}:=\bigcup_{\sigma\in\omega^{\omega}}\bigcap_{i\ge1} B_{\sigma|_i}$.
\end{itemize}
In order to conclude that $X$ is weakly $\mathcal{K}$-analytic, let us show that the subset $\B:=\bigcup_{\sigma\in\omega^{\omega}}\bigcap_{i\ge1} B_{\sigma|_i}$ of $B_X$ and the family $\{B_\sigma\colon \sigma \in \omega^{<\omega}\}$ satisfy condition \ref{it:condition-ii} of Lemma \ref{l:theorem4-FGMZ}. Using that $A_t \subset \A$, for every $t \in \omega^{<\omega}$ and that $A_\emptyset=\A$, it is easy to see that $B_{n^\smallfrown t} \subset \B$, for every $n \in \omega$ and $t \in \omega^{<\omega}$. In order to prove that $\B$ is linearly dense in $X$, we claim that $\{v_n^\alpha\colon \alpha < \kappa, n \in \omega\} \subset \B$. Indeed, fix any $\alpha < \kappa$ and any $n \in \omega$, it follows from \ref{it:nicePRIwAnalytic} that $v_n^\alpha \in \A$ and thus there exists $\sigma\in\omega^{\omega}$ such that $v_{n}^{\alpha}\in\bigcap_{i\ge1}A_{\sigma|_i}$, which implies that $v_{n}^{\alpha}\in\bigcap_{i\ge1}B_{\sigma'|_i}\subset\B$, where $\sigma' \in \omega ^{\omega}$ is defined as $\sigma'(0)=n$ and $\sigma'(i)=\sigma(i-1)$, for every $i\ge 1$. This proves the claim and ensures that $\B$ is linearly dense in $X$, since $\{v_n^\alpha\colon \alpha<\kappa, n \in \omega\}$ is linearly dense in $\bigcup_{\alpha < \kappa} T_\alpha[X]$ and $\bigcup_{\alpha<\kappa}T_\alpha[X]$ is linearly dense in $X$.
Now fix $\varepsilon>0$, $x^*\in X^*\setminus \{0\}$ and $\sigma\in \omega^{\omega}$. Using \ref{it: PRIAjshrinkwAnalytic}, find $i\in\omega$ corresponding to the triple $\varepsilon' :=\tfrac{\varepsilon}{2\|x^*\|}>0$, $x^*\in X^*$ and $\tau\in \omega^{\omega}$, where $\tau(j)=\sigma(j+1)$, for $j\in\omega$. We are going to verify that
\begin{equation}\label{eq:starstarstar}
\vert \{x\in B_{\sigma|_{i+1}}\colon |x^*(x)|>\varepsilon\}\vert<\omega.
\end{equation}

Assume by contradiction that this is not the case. Then using that
$B_{\sigma|_{i+1}}=B_{{\sigma(0)}^\smallfrown {\tau|_i}}=A_{\tau|_i}\cap\{v_{\sigma(0)}^{\alpha}\colon \alpha<\kappa\}$,
we conclude that there exists a strictly increasing sequence $(\alpha_k)_{k \in \omega}$ in $\kappa$ such that $v_{\sigma(0)}^{\alpha_k} \in A_{\tau|_i}$ and $\vert x^*(v_{\sigma(0)}^{\alpha_k}) \vert>\varepsilon$, for every $k \in \omega$. Set $\alpha=\sup_{k}\alpha_{k} \le \kappa$ and note that it follows from the fact that $(Q_\alpha)_{\alpha \le \kappa}$ is $(A_{\tau|_i}, \varepsilon')$-shrinking in $x^*$ that there exists $k \in \omega$ such that $\rho_{A_{\tau|_i}}(Q_{\alpha_k}^*(x^*), Q_\alpha^*(x^*))< 2\varepsilon'\|x^*\|(=\varepsilon)$. In particular, it holds that $\vert Q_{\alpha_k}^*(x^*)(v_{\sigma(0)}^{\alpha_k})-Q_{\alpha}^*(x^*)(v_{\sigma(0)}^{\alpha_k}) \vert <\varepsilon$.
Moreover using that $v_{\sigma(0)}^{\alpha_k} \in (Q_{\alpha_k+1}-Q_{\alpha_k})[X]=Q_{\alpha_k+1}[X] \cap \ker Q_{\alpha_k}$
and $Q_{\alpha_k+1}[X] \subset Q_\alpha[X]$, we conclude that $Q_{\alpha_k}(v_{\sigma(0)}^{\alpha_k})=0$ and $Q_{\alpha}(v_{\sigma(0)}^{\alpha_k})=v_{\sigma(0)}^{\alpha_k}$ and thus we obtain that $\vert x^*(v_{\sigma(0)}^{\alpha_k}) \vert <\varepsilon$, which contradicts the choice of $v_{\sigma(0)}^{\alpha_k}$. We proved  \eqref{eq:starstarstar} and thus verified condition \ref{it:condition-ii} of Lemma \ref{l:theorem4-FGMZ}. Now using this lemma we conclude that $X$ is weakly $\mathcal{K}$-analytic.
\end{proof}

\subsection{Va\v{s}\'ak spaces}

Following \cite{FGMZ}, we say that a Banach space $X$ is \emph{Va\v{s}\'ak} if there exists a family $\{K_n\colon n \in \omega\}$ of weak$^*$-compact subsets of $X^{**}$ such that for every $x\in X$ and $x^{**}\in X^{**}\setminus X$, there exists $n\in\omega$ with $x\in K_n$ and $x^{**}\notin K_n$. The key tools to establish the characterization of Va\v{s}\'ak spaces given in Theorem~\ref{thm:Vasak} are Theorem \ref{thm:constructionSPRI} and the following characterization of Va\v{s}\'ak spaces.

\begin{lemma}[{\cite[Theorem~3]{FGMZ}}]\label{l:theorem3-FGMZ}
For a Banach space $X$, the following conditions are equivalent:
\begin{enumerate}[label = (\roman*)]
    \item $X$ is a Va\v{s}\'ak space;
    \item \label{it:condition-ii} there exist a set $\A \subset B_X$ that is linearly dense in $X$ and a family $\{A_n\colon n \in \omega\}$ of subsets of $\A$ such that for every $\varepsilon>0$, $x^* \in X^*$ and $x \in \A$, there exists $n \in \omega$ such that
    \[x \in A_n \ \text{and} \ \vert \{x' \in A_n\colon \vert x^*(x') \vert >\varepsilon\}\vert< \omega.\]
\end{enumerate}
\end{lemma}

\begin{proof}[Proof of Theorem~\ref{thm:Vasak}]
\ref{it:Xwanalytic}$\Rightarrow$ \ref{it:PSshrinkwanalytic} Let $\A \subset B_X$ and $\{A_n \subset \A\colon n \in \omega\}$ be given by Lemma \ref{l:theorem3-FGMZ} and note that $\A$ countably supports $X^*$. Since $X$ is Plichko, we have that $X$ admits a projectional skeleton $\mathfrak{s}=(P_s)_{s \in \Gamma}$ and thus it follows from Corollary \ref{cor:plichko} that we may assume without loss of generality that $P_s(x) \in \{0, x\}$, for every $x \in \A$ and $s \in \Gamma$.

Clearly condition (a) is satisfied. Now let us show that condition (b) also holds. Fix $\varepsilon>0$, $x^*\in X^* \setminus \{0\}$ and for each $x \in \A$, pick $n(x) \in \omega$ such that
\[x \in A_{n(x)} \ \text{and} \ \vert \{x' \in A_{n(x)}\colon \vert x^*(x') \vert > \varepsilon \Vert x^* \Vert\} \vert<\omega.\]
Set $N:=\{n(x)\colon x \in \A\}$. It is clear that $\A=\bigcup_{j \in N} A_j$ and fixed $j \in N$ the fact that $\mathfrak{s}$ is $(A_j, \varepsilon)$-shrinking in $x^*$ is established with an argument identical to the one presented in the proof of \ref{it:Xwanalytic}$\Rightarrow$ \ref{it:PSshrinkwanalytic} in Theorem~\ref{thm:main}. \\
\ref{it:PSshrinkvasak} $\Rightarrow$ \ref{it:SPRIvasak} If $X$ is separable, then our implication is trivial and in the case when $X$ is nonseparable, our implication follows directly from Theorem \ref{thm:constructionSPRI}. \\
\ref{it:SPRIvasak} $\Rightarrow$ \ref{it:Xvasak} Let $(Q_{\alpha})_{\alpha\leq\kappa}$, $\A$ and $\{A_n\subset B_X\colon n \in \omega\}$ be given by (iii). Note that we may assume that $Q_{\alpha}\neq Q_{\alpha+1}$, for every $\alpha<\kappa$.
For every $\alpha < \kappa$, set $T_\alpha:=Q_{\alpha+1}-Q_\alpha$ and note that since $T_\alpha[X]$ is separable and $T_\alpha[\A]$ is linearly dense in $T_\alpha[X]$, there exists a countable set $\{v_n^\alpha\colon n \in \omega\} \subset T_\alpha[\A] \setminus \{0\}$ that is linearly dense in $T_\alpha[X]$. For each $n,m\in \omega$ define $B_{n,m}:=A_{m}\cap\{v_n^{\alpha}\colon\alpha < \kappa\}$ and set $\B:=\bigcup_{n,m\in\omega} B_{n,m}$.

In order to conclude that $X$ is a Va\v{s}\'{a}k space, let us show that the set $\B$ and the family $\{B_{n,m}\colon n,m \in \omega\}$ satisfy condition \ref{it:condition-ii} of Lemma \ref{l:theorem3-FGMZ}. Using that \ref{it:nicePRIvasak} ensures that $T_\alpha[\A] \setminus \{0\} \subset \A$, for every $\alpha < \kappa$, it is easy to see that $\{v_n^\alpha\colon n \in \omega, \alpha < \kappa\} \subset \B$, which implies that $\B$ is linearly dense in $X$. Now fix $\varepsilon>0$, $x^*\in X^* \setminus \{0\}$ and $x\in \B$. Let $N \subset \omega$ be the set given by \ref{it: PRIAjshrinkvasak} for $\varepsilon':=\frac{\varepsilon}{2 \Vert x^* \Vert}>0$ and $x^*$. Clearly, there exist $\alpha < \kappa$ and $n \in \omega$ such that $x=v_n^\alpha$. Moreover, since $\B \subset \A=\bigcup_{j \in N}A_j$, there exists $j \in N$ such that $x \in A_j$ and thus we have that $x \in B_{n, j}$. In order to conclude the proof, it remains to show that the set $\{x'\in B_{n, j}\colon |x^*(x')|>\varepsilon\}$ is finite. Assume by contradiction that this set is infinite. In this case, there exists a strictly increasing sequence $(\alpha_k)_{k \in \omega}$ in $\kappa$ such that $v_n^{\alpha_k} \in A_j$ and $\vert x^*(v_n^{\alpha_k}) \vert >\varepsilon$, for every $k \in \omega$. Since $(Q_\alpha)_{\alpha\le \kappa}$ is $(A_j, \varepsilon')$-shrinking in $x^*$, arguing as in the proof of \ref{it:SPRIwanalytic} $\Rightarrow$ \ref{it:Xwanalytic} in Theorem~\ref{thm:main}, we achieve a contradiction and establish our implication.
\end{proof}

\subsection{WCG spaces and their subspaces}

Recall that a Banach space is said to be \emph{weakly compactly generated} (WCG) if it contains a linearly dense and weakly compact subset. The key tools to establish the characterization of WCG spaces given in Theorem~\ref{thm:WCG} are Theorem \ref{thm:constructionSPRI} and the following characterization of WCG spaces.

\begin{lemma}[{\cite[Theorem~1]{FGMZ}}]\label{l:theorem1-FGMZ}
For a Banach space $X$, the following conditions are equivalent
\begin{enumerate}[label = (\roman*)]
    \item $X$ is a WCG space;
    \item \label{it:condition-ii-Thm1-Inner} there exists a set $\A \subset B_X$ that is linearly dense in $X$ and such that for every $\varepsilon>0$ and for every $x^* \in X^*$ it holds that
    \[\vert \{x \in \A\colon \vert x^*(x) \vert>\varepsilon\}\vert< \omega.\]
\end{enumerate}
\end{lemma}

Note that the equivalence between \ref{it:XWCG} and \ref{it:PSshrinkWCG} in Theorem \ref{thm:WCG} can be deduced also from the methods of proof from \cite[Theorem 21]{FM18}; however, our argument is shorter (using the results proved above) and condition \ref{it:SPRIWCG} is new.

\begin{thm}\label{thm:WCG}
Let $X$ be a Banach space and put $\kappa:=\dens(X)$. Then the following conditions are equivalent.
\begin{enumerate}[label = (\roman*)]
    \item\label{it:XWCG} $X$ is a WCG space.
    \item\label{it:PSshrinkWCG} There exist a projectional skeleton $\mathfrak{s}=(P_s)_{s\in\Gamma}$ on $X$ and a non-empty set $\A\subset B_X$, satisfying the following conditions:
\begin{enumerate}
    \item\label{it:lindenscountsupportWCG}  $\A$ is linearly dense in $X$ and countably supports all of $X^*$;
    \item $\mathfrak{s}$ is $(\A,0)$-shrinking in every element of $X^*$.
\end{enumerate}
\item\label{it:SPRIWCG} There exist a SPRI $(Q_{\alpha})_{\alpha\leq \kappa}$ in $X$ and a non-empty set $\A\subset B_X$, satisfying the following conditions:
\begin{enumerate}[label=(3\alph*)]
\item $\A $ is linearly dense in $X$ and countably supports $X^*$;
\item\label{it:nicePRIWCG}  for every $x\in\A$ it holds that $\{Q_\alpha(x)\colon \alpha\leq \kappa\}\subset \{0,x\}$;
\item\label{it: PRIAjshrinkWCG} $(Q_\alpha)_{\alpha\leq \kappa}$ is $(\A,0)$-shrinking in every element of $X^*$.
\end{enumerate}
\end{enumerate}
\end{thm}

\begin{proof}
\ref{it:XWCG}$\Rightarrow$\ref{it:PSshrinkWCG} Let $\A$ be the set given by Lemma \ref{l:theorem1-FGMZ} and note that $\A$ countably supports $X^*$. Since $X$ is (trivially) Plichko, we have that $X$ admits a projectional skeleton $\mathfrak{s}=(P_s)_{s \in \Gamma}$ and thus it follows from Corollary \ref{cor:plichko} that we may assume without loss of generality that $P_s(x) \in \{0, x\}$, for every $x \in \A$ and $s \in \Gamma$. Finally, given any $x^* \in X^*$, the fact that $\mathfrak{s}$ is $(\A,0)$-shrinking in $x^*$ can be easily checked using that for every $\varepsilon>0$, the set $\{x \in \A\colon \vert x^*(x) \vert>\varepsilon\}$ is finite, similarly as argued in the proof of \ref{it:Xwanalytic}$\Rightarrow$ \ref{it:PSshrinkwanalytic} in Theorem~\ref{thm:main}. \\
\ref{it:PSshrinkWCG}$\Rightarrow$ \ref{it:SPRIWCG} If $X$ is separable, then \ref{it:SPRIWCG} trivially holds and in the case when $X$ is nonseparable, the implication follows directly from Theorem \ref{thm:constructionSPRI}. \\
\ref{it:SPRIWCG}$\Rightarrow$\ref{it:XWCG} Let $(Q_{\alpha})_{\alpha\leq\kappa}$ and $\A$ be given by (iii). Note that we may assume that $Q_{\alpha}\neq Q_{\alpha+1}$, for every $\alpha<\kappa$. For every $\alpha < \kappa$, set $T_\alpha:=Q_{\alpha+1}-Q_\alpha$. Since $T_\alpha[X]$ is separable and $T_\alpha[\A]$ is linearly dense in $T_\alpha[X]$, there exists a countable set $\{v_n^\alpha\colon n \ge 1\} \subset T_\alpha[\A] \setminus \{0\}$ that is linearly dense in $T_\alpha[X]$. In order to conclude that $X$ is WCG, we will show that the set $\B:=\bigcup_{\alpha<\kappa}\{\tfrac{v_n^{\alpha}}{n}\colon n\ge 1\}$ satisfies condition \ref{it:condition-ii-Thm1-Inner} of Lemma \ref{l:theorem1-FGMZ}. Note that $\B \subset B_X$, since it follows from condition \ref{it:nicePRIWCG} that $T_\alpha[\A] \setminus \{0\} \subset \A$. Moreover the linear density of $\B$ in $X$ follows from the fact that $\B$ is linearly dense in $\bigcup_{\alpha < \kappa} T_\alpha[X]$ and $\bigcup_{\alpha < \kappa} T_\alpha[X]$ is linearly dense in $X$. Now fix any $\varepsilon>0$, and any $x^*\in X^*$. We will show that the set $\{x\in \B\colon |x^*(x)|>\varepsilon \}$ is finite. Assume by contradiction that this is not the case. If there are only finitely many $\alpha < \kappa$ such that $\{x\in \B\colon |x^*(x)|>\varepsilon \} \cap T_\alpha[X] \ne \emptyset$, we easily arrive to a contradiction, since $\lim_{n \to \infty} \tfrac{v_n^\alpha}{n}=0$, for every $\alpha < \kappa$.
Otherwise, there exist a strictly increasing sequence $(\alpha_k)_{k \in \omega}$ in $\kappa$ and a sequence of nonzero natural numbers $(n_k)_{k \in \omega}$ such that
\[\forall k \in \omega\colon\quad \Big\vert x^*\big(\tfrac{v_{n_k}^{\alpha_k}}{n_k}\big) \Big\vert> \varepsilon.\]
Since $(Q_\alpha)_{\alpha\le \kappa}$ is $(\A,0)$-shrinking in $x^*$, arguing as in the proof of \ref{it:SPRIwanalytic} $\Rightarrow$ \ref{it:Xwanalytic} in Theorem~\ref{thm:main}, we obtain that there exists $k \in \omega$ such that $\vert x^*(v_{n_k}^{\alpha_k}) \vert < \varepsilon$, since $v_{n_k}^{\alpha_k} \in \A$. But this contradicts the choice of $v_{n_k}^{\alpha_k}$ and thus establishes the implication.
\end{proof}

It is worth mentioning that a subspace of a WCG space is not necessarily WCG (see \cite{Ros74}). The key tools to establish the characterization of subspaces of WCG spaces given in Theorem~\ref{thm:SWCG} are Theorem \ref{thm:constructionSPRI} and the following characterization of subspaces of WCG spaces.

\begin{lemma}[{\cite[Theorem~2]{FGMZ}}]\label{l:theorem2-FGMZ}
For a Banach space $X$, the following conditions are equivalent
\begin{enumerate}[label = (\roman*)]
    \item $X$ is a subspace of a WCG space;
    \item \label{it:condition-ii-Thm2-Inner} there exists a set $\A \subset B_X$ that is linearly dense in $X$ and such that for every $\varepsilon>0$ there exists a decomposition $\A=\bigcup_{n \in \omega}A_n^\varepsilon$ satisfying the following condition
    \[\forall n \in \omega, \ \forall x^* \in X^* \colon \; \vert \{x \in A_n^\varepsilon\colon \vert x^*(x) \vert>\varepsilon\}\vert< \omega.\]
\end{enumerate}
\end{lemma}

Note that the equivalence between \ref{it:XSWCG} and \ref{it:PSshrinkSWCG} in Theorem \ref{thm:SWCG} can be deduced also from the methods of proof from \cite[Theorem 24]{FM18}; however, our argument is shorter (using the results proved above) and  condition \ref{it:SPRISWCG} is new.

\begin{thm}\label{thm:SWCG}
Let $X$ be a Banach space and put $\kappa:=\dens(X)$. Then the following conditions are equivalent.
\begin{enumerate}[label = (\roman*)]
    \item\label{it:XSWCG} $X$ is isometric to a subspace of a WCG space.
    \item\label{it:PSshrinkSWCG} There exist a projectional skeleton $\mathfrak{s}=(P_s)_{s\in\Gamma}$ on $X$ and a non-empty set $\A\subset B_X$, satisfying the following conditions:
\begin{enumerate}
    \item\label{it:lindenscountsupportSWCG}   $\A $ is linearly dense in $X$ and countably supports $X^*$;
    \item for every $\varepsilon>0$, there exists a decomposition $\A=\bigcup_{n\in\omega}A_n^{\varepsilon}$ such that $\mathfrak{s}$ is $(A_n^{\varepsilon},\tfrac{\varepsilon}{\Vert x^* \Vert})$-shrinking in $x^*$, for every $n\in\omega$ and every $x^*\in X^* \setminus \{0\}$.
\end{enumerate}
\item\label{it:SPRISWCG} There exist a SPRI $(Q_{\alpha})_{\alpha\leq \kappa}$ in $X$ and a non-empty set $\A\subset B_X$, satisfying the following conditions:
\begin{enumerate}[label=(3\alph*)]
\item $\A $ is linearly dense in $X$ and countably supports $X^*$;
\item\label{it:nicePRISWCG}  for every $x\in\A$ it holds that $\{Q_\alpha(x)\colon \alpha\leq \kappa\}\subset \{0,x\}$;
\item\label{it: PRIAjshrinkSWCG} for every $\varepsilon>0$, there exists a decomposition $\A=\bigcup_{n\in\omega}A_n^{\varepsilon}$ such that $(Q_\alpha)_{\alpha\leq\kappa}$ is $(A_n^{\varepsilon},\tfrac{\varepsilon}{\Vert x^* \Vert})$-shrinking in $x^*$, for every $n\in\omega$ and every $x^*\in X^* \setminus \{0\}$.
\end{enumerate}
\end{enumerate}
\end{thm}

\begin{proof}
\ref{it:XWCG}$\Rightarrow$\ref{it:PSshrinkSWCG} Let $\A$ be the set given by Lemma \ref{l:theorem2-FGMZ} and note that $\A$ countably supports $X^*$. Since $X$ is Plichko, we have that $X$ admits a projectional skeleton $\mathfrak{s}=(P_s)_{s \in \Gamma}$ and thus it follows from Corollary \ref{cor:plichko} that we may assume without loss of generality that $P_s(x) \in \{0, x\}$, for every $x \in \A$ and $s \in \Gamma$. Now fix $\varepsilon>0$, and let $\A=\bigcup_{n \in \omega} A_n^\varepsilon$ be the decomposition given by Lemma \ref{l:theorem2-FGMZ}. The fact that $\mathfrak{s}$ is $(A_n^\varepsilon,\tfrac{\varepsilon}{\Vert x^* \Vert})$-shrinking in $x^*$, for every $n \in \omega$ and $x^* \in X^* \setminus \{0\}$ can be easily checked using that the set $\{x \in A_n^\varepsilon\colon \vert x^*(x) \vert>\varepsilon\}$ is finite, similarly as argued in the proof of \ref{it:Xwanalytic}$\Rightarrow$ \ref{it:PSshrinkwanalytic} in Theorem~\ref{thm:main}. \\
\ref{it:PSshrinkSWCG}$\Rightarrow$ \ref{it:SPRISWCG} If $X$ is separable, then \ref{it:SPRISWCG} trivially holds and in the case when $X$ is nonseparable, the implication follows directly from Theorem \ref{thm:constructionSPRI}. \\
\ref{it:SPRISWCG}$\Rightarrow$\ref{it:XSWCG} Let $(Q_{\alpha})_{\alpha\leq\kappa}$ and $\A$ be given by (iii). Note that we may assume that $Q_{\alpha}\neq Q_{\alpha+1}$, for every $\alpha<\kappa$.
For every $\alpha < \kappa$, set $T_\alpha:=Q_{\alpha+1}-Q_\alpha$ and note that since $T_\alpha[X]$ is separable and $T_\alpha[\A]$ is linearly dense in $T_\alpha[X]$, there exists a countable set $\{v_n^\alpha\colon n \ge 1\} \subset T_\alpha[\A] \setminus \{0\}$ that is linearly dense in $T_\alpha[X]$. In order to conclude that $X$ is a subspace of a WCG space, we will show that the set $\B:=\bigcup_{\alpha<\kappa}\{v_n^{\alpha}\colon n\in\omega\}$ satisfies condition \ref{it:condition-ii-Thm2-Inner} of Lemma \ref{l:theorem2-FGMZ}. Note that $\B \subset B_X$, since it follows from condition \ref{it:nicePRIWCG} that $T_\alpha[\A] \setminus \{0\} \subset \A$. Moreover the linear density of $\B$ in $X$ follows from the fact that $\B$ is linearly dense in $\bigcup_{\alpha < \kappa} T_\alpha[X]$ and $\bigcup_{\alpha < \kappa} T_\alpha[X]$ is linearly dense in $X$.

Now fix $\varepsilon>0$, set $\varepsilon':=\varepsilon/2$ and let $\A=\bigcup_{n \in \omega} A_n^{\varepsilon'}$ be the decomposition given by \ref{it: PRIAjshrinkSWCG}. For each $n,m \in \omega$, define $B_{n,m}^{\varepsilon}:=A_{m}^{\varepsilon'}\cap \{v_n^{\alpha}: \alpha <\kappa\}$ and note that $\B=\bigcup_{n,m\in\omega}B_{n,m}^{\varepsilon}$, since $\B \subset \A$. In order to conclude the proof, it remains to show that fixed $n, m \in \omega$ and $x^* \in X^* \setminus \{0\}$, the set $\{x\in B_{n,m}^{\varepsilon}\colon |x^*(x)|>\varepsilon\}$ is finite. Assume by contradiction that this set is infinite. In this case, there exits a strictly increasing sequence $(\alpha_k)_{k \in \omega}$ in $\kappa$ such that $v_n^{\alpha_k} \in A_m^{\varepsilon'}$ and $\vert x^*(v_n^{\alpha_k}) \vert >\varepsilon$, for every $k \in \omega$. Since $(Q_\alpha)_{\alpha\le \kappa}$ is $(A_m^{\varepsilon'}, \tfrac{\varepsilon'}{\Vert x^* \Vert})$-shrinking in $x^*$, arguing as in the proof of \ref{it:SPRIwanalytic} $\Rightarrow$ \ref{it:Xwanalytic} in Theorem~\ref{thm:main}, we achieve a contradiction and establish our implication.
\end{proof}

\section{Open problems and remarks}\label{sec:questions}

As mentioned previously, we found the relationship between SPRI and projectional skeletons quite interesting. Observe that combining an inductive argument with Proposition \ref{prop:pri} and a result similar to \cite[Proposition 6.2.7]{FabRedBook} (just replacing the assumption that the sequence of projections is a PRI by the assumption that it satisfies conditions \ref{it:P0-Pkappa}, \ref{it:comutative}, \ref{it:upDirectedLimit},  \ref{it:skeletonOnSubset} and \ref{it:subskeletonOnDifferences} of Proposition \ref{prop:pri}), one can easily construct a SPRI in any Banach space with a projectional skeleton.
We would like to understand better the properties of the SPRI built from a given projectional skeleton. In particular, this could be useful when dealing with questions proposed in \cite[Section 6]{kal20}, where the author already obtained many deep results concerning the connection between projectional skeletons and PRI. As an example of a problem, we suggest the following.

\begin{problem}
Find a property $(P)$ of SPRI such that a Banach space $X$ has a projectional skeleton if and only if it admits a SPRI satisfying $(P)$.
\end{problem}

It is important to note that the methods used in the present paper depend heavily on Corollary~\ref{cor:plichko}, that is, on the fact that we work with subclasses of Plichko spaces. We observe that given a Plichko space $X$, it follows from \cite[Theorem~3.1]{kal20} that $X$ admits a projectional skeleton $\mathfrak{s}=(P_s)_{s \in \Gamma}$ and a a Markushevich basis $(x_i,x_i^*)_{i\in I}$ such that $\{x_i\colon i \in I\}$ countably supports $D(\mathfrak{s})$ and therefore Theorem \ref{thm:constructionSPRI} ensures that $X$ admits a SPRI $(Q_\alpha)_{\alpha \le \dens X}$ such that $Q_\alpha[\{x_i\colon i \in I\}] \subset \{x_i\colon i \in I\} \cup \{0\}$, for every $\alpha \le \dens X$. In this context, it is natural to propose Question \ref{q:Mbasis-invariant-skeleton}, that seems to be related to \cite[Question 6.3]{kal20}.

\begin{question}\label{q:Mbasis-invariant-skeleton}
Let $X$ be a Banach space with a projectional skeleton. Does $X$ admit a Markushevich basis $(x_i,x_i^*)_{i\in I}$ and a SPRI $(Q_\alpha)_{\alpha\leq \dens X}$ such that $Q_\alpha[\{x_i\colon i\in I\}]\subset \{x_i\colon i\in I\} \cup \{0\}$, for every $\alpha\leq \dens X$?
\end{question}

A final suggestion for further research is to characterize other subclasses of Plichko spaces that were considered in \cite{FGMZ}. For instance, we suggest the following.

\begin{problem}
To characterize the class of Hilbert generated spaces using projectional skeletons and SPRI.
\end{problem}

\noindent\textbf{Acknowledgments.}\enspace The authors would like to thank Marián Fabian for carefully reading this manuscript and making valuable suggestions.

\bibliography{ref}{}
\bibliographystyle{siam}

\end{document}